\documentclass[epsf,a4paper,11pt]{article}
 \usepackage{fullpage}
\usepackage{amssymb,latexsym,amsmath}
\usepackage{graphicx}
\usepackage{comment}
\usepackage{url}
\usepackage{amscd}
\usepackage{amsthm}
\usepackage{fontenc}
\usepackage[utf8x]{inputenc}
\newtheorem{theorem}{Theorem}%[section]
\newtheorem{coro}{Corollary}
\newtheorem{lem}{Lemma}
\newtheorem{lemma}{Lemma}

\newtheorem{definition}{Definition}
\newtheorem{remark}{Remark}
\newtheorem{proposition}{Proposition}
\newtheorem{example}{Example}
\include{defs}
\def\Xint#1{\mathchoice
{\XXint\displaystyle\textstyle{#1}}%
{\XXint\textstyle\scriptstyle{#1}}%
{\XXint\scriptstyle\scriptscriptstyle{#1}}%
{\XXint\scriptscriptstyle\scriptscriptstyle{#1}}%
\!\int}
\def\XXint#1#2#3{{\setbox0=\hbox{$#1{#2#3}{\int}$ }
\vcenter{\hbox{$#2#3$ }}\kern-.6\wd0}}
\newcommand{\eps}{\varepsilon} \newcommand{\TT}{{\mathbb T^2_\ell}}

\def\dashint{\Xint-}
\def\eps{\varepsilon}
\def\({\left(}
\def\){\right)}
\def\1{\mathbf{1}}
\newcommand{\lessim}{\stackrel{<}{\sim}}

\def\div{\mathrm{div}}

\def\dt0{{{\frac{d}{dt}}_{|t=0}}}

\def\ep{\varepsilon}

\def\hal{\frac{1}{2}}

\def\h{{\eta}}

\def\lep{{|\mathrm{log }\ \ep|}}

\def\l|{\left|}

\def\mr{\mathbb{R}}

\def\nab{\nabla}

\def\r|{\right|}

\def\UR{{ K_R}}

\def\h{h_{\epsilon}^{\prime}}
\def\eps{{\varepsilon}}
\title{Asymptotics of non-minimizing stationary points of the Ohta-Kawasaki energy and its sharp interface version}
\author{Dorian Goldman \thanks{University of Cambridge (DPMMS) - dg443@dpmms.cam.ac.uk}}

\begin{document}

\maketitle

\begin{abstract}
We study a non-local Cahn-Hilliard energy arising in the study of di-block copolymer melts, often referred to as the Ohta-Kawasaki energy in that context. In this model,
two phases appear, which interact via a Coulombic energy. As in \cite{gms10b}--\cite{gms11b}, we focus on the regime where one of the phases has a very small
volume fraction, thus creating ``droplets'' of the minority phase in a ``sea'' of the majority phase. In this paper, we address the asymptotic behavior of non-minimizing stationary points in dimensions $n \geq 2$ left open by the study of the $\Gamma$-convergence of
the energy established in \cite{gms10b}--\cite{gms11b}, which provides information only for almost minimizing sequences when $n=2$. In particular, we prove that (asymptotically) stationary points satisfy a force balance condition which implies that the minority 
phase distributes itself uniformly in the background majority phase. Our proof uses and generalizes the framework of Sandier-Serfaty \cite{SS5,SS3}, used in the context of stationary points of the Ginzburg-Landau model, to 
higher dimensions. When $n=2$, using the regularity results obtained in \cite{GV}, we also are able to conclude that the droplets in the sharp interface energy become asymptotically round when the number of droplets is constrained
to be finite and have bounded isoperimetric deficit. 
\end{abstract}

\section{Introduction}

This paper is devoted to the convergence of stationary points of
the Ohta-Kawasaki energy functional \cite{ohta86} in the small volume regime. The energy functional has the following form:
\begin{align}
  \label{EE}
 \mathcal{E}[u] = \int_\Omega \( \frac{\eps^2}{2} |\nabla u|^2 +
  V(u)\) dx + \frac{1}{2} \int_\Omega \int_\Omega (u(x)- \bar u)
  G_0(x, y) (u(y)- \bar u) \, dx \, dy,
\end{align}
where $\Omega$ is the domain occupied by the material, $u: \Omega \to
\mathbb R$ is the scalar order parameter, $V(u)$ is a symmetric
double-well potential with minima at $u = \pm 1$, such as the usual
Ginzburg-Landau potential $V(u) = \tfrac{1}{4} (1 - u^2)^2$, $\eps > 0$ is a parameter
characterizing interfacial thickness, $\bar u \in (-1, 1)$ is the
background charge density, and $G_0$ is the Neumann Green's function
of the Laplacian, i.e., $G_0$ solves
\begin{align}
  \label{G0}
  -\Delta G_0(x, y) = \delta(x - y) - {1 \over |\Omega|}, \qquad
  \int_\Omega G_0(x, y) \, dx = 0,
\end{align}
where $\Delta$ is the Laplacian in $x$ and $\delta(x)$ is the Dirac
delta-function, with Neumann boundary conditions. Note that $u$ is
also assumed to satisfy the ``charge neutrality'' condition
\begin{align}
  \label{neutr}
  {1 \over |\Omega|} \int_\Omega u \, dx = \bar u.
\end{align}

For a discussion of the motivation and the main quantitative features
of this model, see  \cite{gms10b}, as well as
\cite{m:cmp10,m:pre02}. For specific applications to physical systems,
we refer the reader to
\cite{degennes79,stillinger83,ohta86,nyrkova94,glotzer95,lundqvist,m:pre02,m:phd}. For the remainder of this paper we focus on the case
where $\Omega$ is the flat n-dimensional torus $\mathbb{T}^n = [0,1)^n$ with periodic boundary conditions, unless otherwise specified.

\noindent We focus most of our attention on the following ``sharp interface'' version of \eqref{EE}:
\begin{equation}\label{E2}
 E^{\epsilon}[u] = \frac{\eps}{2} \int_{\mathbb{T}^n} |\nabla u| + \frac{1}{2}\iint_{\mathbb{T}^n \times \mathbb{T}^n} (u(x)-\bar u) G(x-y) (u(y)-\bar u)dxdy,
\end{equation}
reserving most of our analysis of the diffuse interface energy \eqref{EE} to the final section of this paper (Section \ref{sec5}). In \eqref{E2}, $G$ is the \emph{screened} Poisson kernel solving
\begin{equation}
 -\Delta G + \kappa^2 G = \delta(x-y) \textrm{ in } \mathbb{T}^n,
\end{equation}
for $\kappa =1/\sqrt{V''(1)} > 0$ and $u \in \mathcal{A}$ where
\begin{equation}
 \mathcal{A} := \{ u \in BV(\mathbb{T}^n; \{-1,1\})\}.
\end{equation}

The charge neutrality condition (cf. equation \eqref{neutr}) is no longer imposed, i.e. $\int_{\mathbb{T}} u \neq \bar u$. This is related to the fact that the charge of the minority
phase is expected to partially redistribute itself into the majority phase to ensure screening of the induced non-local field (see \cite{gms10b} for a more detailed discussion).
 The energy \eqref{E2} was first studied
in \cite{m:cmp10} where the connection between \eqref{EE} and \eqref{E2} is made for exact minimizers. Moreover, when $n=2$, it is shown that when $\bar u$ is close to $-1$, minimizers of \eqref{E2}
form almost spherical ``droplets'' of the minority phase $\{u=+1\}$ with the same radius, distributed uniformly throughout the domain. In \cite{gms10b}--\cite{gms11b} the full $\Gamma$-limit of \eqref{E2} was computed
to first and second order near the onset of non-trivial minimizers (see \cite{Braides} for an introduction to $\Gamma$-convergence), with \cite{gms10b} addressing the $\Gamma$ limit of \eqref{EE} as well. There it is shown that, in addition, almost minimizers
form (on average) almost spherical droplets of the phase $\{u=+1\}$, with almost the same radius and which are once again distributed uniformly throughout the domain. An important observation in these works is that, as $\eps \to 0$, the number of disjoint connected components of $\{u^{\eps}=+1\}$ may be unbounded \cite{gms10b, gms11b, m:cmp10}, and
the results can thus be seen as generalizations of the work of Choksi and Peletier who study a suitably rescaled version of \eqref{EE} and \eqref{E2} in the absence of screening (ie. $\kappa = 0$) and when the number of droplets is constrained to be finite \cite{choksi10,choksi11}. More precisely,
they compute the $\Gamma$-limit in this setting of \eqref{EE} and \eqref{E2} in \cite{choksi11} and \cite{choksi10} respectively, showing, in particular, that the droplets of the minority phase $\{u=+1\}$ shrink to points whose magnitudes and locations
are determined via a limiting Coulombic interaction energy. A related result concerning minimizers
is the work of Alberti-Choksi-Otto and Spadaro \cite{ACO,spadaro09}, wherein it is shown that the energy of minimizers of \eqref{EE} and \eqref{E2} respectively is uniformly distributed throughout the domain.

All of the above results are concerned with \emph{minimizing} stationary points of the energies \eqref{EE} and \eqref{E2}. Moreover, all of the results regarding the asymptotics of minimizers when the number of droplets is unbounded work only in dimension $n=2$. In this paper we address a question left open in the above analysis
which is that of the asymptotic behavior of a priori \emph{non-minimizing} stationary points of the energies \eqref{EE} and \eqref{E2} which, moreover, applies to any dimension $n \geq 2$. There has
been some work in this context by  R\"{o}ger and Tonegawa \cite{CPoints}. They show that when the number of droplets is constrained to be finite in a bounded domain $\Omega$ with a fixed volume fraction, 
that any sequence of critical points $(u^{\eps})_{\eps}$ of \eqref{EE}, i.e. solutions to

 \[ -\eps^2 \Delta u^{\eps} + V'(u^{\eps}) + \phi_{\eps} = \lambda_{\eps},\]

where $\phi_{\eps}(x) = \(G_0(x-\cdot) * (u^{\eps}-\bar u^{\eps})\)(x)$ and $\lambda_{\eps}$ is a Lagrange multiplier arising from \eqref{neutr}, satisfying mild bounds on the energy, converge in an appropriate sense to the Gibbs-Thompson law:
\begin{equation}\label{cpweak} \sigma H =\left\{\begin{array}{ccc}
-\phi + \lambda&\mbox{for}&x \in \partial^*\{u=+1\}\\
0&\mbox{for}&x \in \partial \{u=+1\} \backslash \partial^* \{u=+1\}.
\end{array}\right. \end{equation}

Here $H$ is the mean curvature of $\{u=+1\}$ where $u \in BV(\Omega; \{-1,+1\})$ and $\phi$ are both appropriately rescaled limits of $u^{\eps}$ and $\phi_{\eps}$ respectively, $\sigma$ is an integer which arises from the `folding' of the interfaces, $\partial^*\{u=+1\}$ denotes the reduced
boundary of $\{u=+1\}$ (see Section \ref{finper}) and $\lambda$ is the limiting Lagrange multiplier constant.
\begin{comment}
 More precisely, if $u^{\eps}$ solves
\[ -\eps \Delta u^{\eps} + \frac{1}{\eps} V'(u^{\eps}) + \phi_{\eps} = \lambda_{\eps}.\]
Then the limiting interfaces defined by $\mu_{\eps} = \eps |\nabla u^{\eps}|^2 + \frac{1}{\eps} V(u^{\eps})d\mathcal{L}^n$, where $\mathcal{L}^n$ is Lebesgue measure, converge in the varifold
sense to
\end{comment}
This establishes the connection between critical points of the diffuse
interface energy \eqref{EE} and its sharp interface analogue (replacing the first Cahn-Hilliard term with perimeter). Our goal differs from that
of \cite{CPoints}, as we wish to establish the distribution of the small droplets in the regime where the volume of the minority phase vanishes, and the
number of droplets is not constrained to be finite a priori for \eqref{EE} and \eqref{E2}. 

 To understand our goal more precisely, we recall some of the main results of \cite{gms10b} for almost minimizers of \eqref{E2}. We begin by setting 
\[ \bar u^{\eps} = -1 + \delta(\eps),\]
in \eqref{E2} and show for almost minimizers of \eqref{E2}, when $\delta(\eps) = \eps^{2/3}|\ln \eps|^{1/3}\bar \delta$ with $\bar \delta >0$, that the number
of droplets of $\{u^{\eps}=+1\}$ is $O(|\ln \eps|)$ as $\eps \to 0$ and, moreover, that
\begin{equation}\label{msrconv1} \omega_{\eps} := \bar \delta \delta(\eps)^{-1} (1 + u^{\eps}) \rightharpoonup \bar \omega \textrm{ in } C(\mathbb{T}^2)^*,\end{equation}
where $\omega_{\eps}$ is the ``normalized droplet density'' of the phase $\{u^{\eps}=+1\}$ and where $\bar \omega$ is the unique constant density minimizer to 
\begin{align}
  \label{E0mu}
  E^0[\omega ] = \frac{\bar \delta^2}{2\kappa^2} + \left(3^{2/3} -
    \frac{2 \bar \delta}{\kappa^2} \right) \int_{\mathbb{T}^2} d\omega  + 2
  \iint_{\mathbb{T}^2 \times \mathbb{T}^2} G(x - y) d \omega (x) d \omega (y),
\end{align}
over all Radon measures $\omega \in H^{-1} (\mathbb{T}^2)$. Moreover, $\bar \omega$ is given explicitly by
\begin{align}
  \label{mubar}
  \bar\omega  = \max\(\tfrac12 (\bar\delta - \bar \delta_c), 0\) \qquad \text{with}
  \qquad E^0[\bar\omega] = \tfrac{\bar\delta_c}{2 \kappa^2} (2 \bar\delta
  - \bar\delta_c),
\end{align}
where $\bar \delta_c > 0$ is the critical volume fraction for the onset of non-trivial minimizers (ie. $\bar \omega \neq 0$).
In addition, setting $v_{\eps}$ to be the solution to
\[ -\Delta v_{\eps} + \kappa^2 v_{\eps} = \omega_{\eps},\]
we conclude that 
\begin{equation}\label{potconv} \nabla v_{\eps} \rightharpoonup 0 \textrm{ weakly in } H^{1}(\mathbb{T}^2).\end{equation}

The convergence in equations \eqref{msrconv1} and \eqref{potconv} show that $v_{\eps}$ and $\omega_{\eps}$ are asymptotically constant in an averaged sense as $\eps \to 0$, which physically
suggests the droplets are uniformly distributed throughout the domain. One way of phrasing the goal of this paper, is to ask the following question: Do the normalized droplet densities still converge weakly to a constant when we drop the assumption of minimality? Moreover, does this fact continue to hold in higher dimensions? We answer these questions under the single assumption that the perimeter of the set $\{u^{\eps}=+1\}$ vanishes as $\eps \to 0$. 
%More precisely, we derive (asymptotically) a ``force balance" condition on the limiting measure $\omega$ which, in particular, implies that $\omega$ is constant if it is regular enough, and a ``vanishing gradient property" in the case of a finite number of droplets. 
We make similar conclusions for the diffuse interface energy \eqref{EE}, but reserve
this discussion for a separate section (Section 7).

Before we proceed, we give a precise definition of a stationary point of \eqref{E2}. In addition to the class $\mathcal{A}$ defined above, we occasionally consider stationary points in $\mathcal{A}$ with mass constraint $m$:
\begin{equation}\label{Amm}
 \mathcal{A}_m := \left\{u \in \mathcal{A} : \int_{\mathbb{T}^n} u = m\right\}.
\end{equation}

\begin{definition}\label{defcp}
 A function $u \in \mathcal{A}$ is said to be a \emph{stationary point of \eqref{E2} in $\mathcal{A}$} if for any  $C^1$ vector field $X:\mathbb{T}^n \to \mathbb{R}^n$ we have,
setting $\phi_t(x) = x + tX(x)$, that
\begin{align}\label{dervan} \frac{d}{dt}\Big|_{t=0}E^{\eps}(u \circ \phi_t) = 0.\end{align}
If \eqref{dervan} holds only for all $\phi_t$ such that $u \circ \phi_t \in \mathcal{A}_m$ for all $t$ sufficiently small, then we call $u$
a \emph{stationary point of \eqref{E2} in $\mathcal{A}_m$}.
\end{definition}

We proceed by showing that, away from a very small set on which the droplets are concentrated, we obtain a limiting condition on the measure $\omega_{\eps}$ which takes the form
\begin{equation}\label{cpeqn} \omega_{\eps} \nabla v^{\eps} \to \omega \nabla v = 0,\end{equation}
in a suitably weak sense. The convergence above clearly does not follow from the weak convergence
of $\omega_{\eps}$ (cf. equation \eqref{msrconv1}) and the weak convergence of the potential (cf. equation \eqref{potconv}). This is similar to the problem which
arises when studying weak limits of solutions to the Euler equations in vorticity form as in \cite{pernamajda, pernamajda2, Chemin, delort, Majda,zheng} in dimension $n=2$. It was originally 
Delort \cite{delort} who first recognized the phenomenon of ``vorticity concentration cancellation", which allows one to nonetheless pass to the limit
in \eqref{cpeqn} in a distributional sense when $\omega_{\eps}$ has a distinguished sign. Similar analysis was done by DiPerna and Majda \cite{pernamajda,pernamajda2,Majda} which
which allows for $\omega$ to have mixed signs under additional assumptions. There are, of course, natural regularity issues with the above equation, and we will see in Theorem \ref{main2} that the regularity we assume on $\omega$
allows us to obtain more precise information from \eqref{cpeqn}. When $\omega$ is a smooth density for instance, it is easy to see that \eqref{cpeqn} implies
that $\omega$ is constant on $\mathbb{T}^n$, so that the normalized droplet densities converge weakly to a constant. We obtain two characterizations of this condition, both of which imply that the droplets of the minority phase satisfy a kind of ``force balance" condition, where the overall force on each limiting droplet is zero. Moreover, unlike the analysis of the Euler equations, our approach applies to all dimensions $n \geq 2$.

\medskip

We have the \emph{additional difficulty}, however, that we have contributions from the local terms in \eqref{EE} and \eqref{E2} which measure the perimeter
of the level sets of $u$ when we take variations of the energy. Here we adopt, and generalize, the techniques in \cite{SS5} and \cite[Chapter 13]{SS3} which were used to prove similar results in the context
of Ginzburg-Landau. There it is shown that it suffices to establish \eqref{cpeqn} away from a very small set where the contributions of the surface
terms are concentrated. Thus this framework can be seen as a generalization of the method of vorticity concentration cancellation introduced by
Delort \cite{delort} for measures with distinguished sign, which allows for additional contributions to \eqref{cpeqn} that are concentrated on small sets, and which also allows for the measures to take on mixed signs, making it somewhat more similar to the work of DiPerna and Majda \cite{pernamajda,pernamajda2, Majda}.

\medskip

In order to make sense of \eqref{E2} and its first variation, we must use extensively the theory of sets of finite perimeter (see \cite{maggi,Gariepy} for nice expositions, or \cite{Simons} for a more general treatment which includes varifolds, which may have higher co-dimension). Unlike the analysis of the corresponding Euler-Lagrange equation which corresponds
to minimizers in \cite{m:cmp10}, here we will assume no minimality, and thus cannot expect global smoothness of the boundary. While it is known that local minimizers have boundaries which are of class $C^{3,\alpha}$ for some $\alpha>0$ \cite{ambrosio3,massari,tamanini,m:cmp10,sternberg}, the question of regularity of the reduced boundary of stationary points of \eqref{E2} has only recently been addressed in \cite{GV}. More precisely, in \cite{GV}, we provide a simple proof that the reduced boundary of any stationary point of \eqref{E2} is of class $C^{3,\alpha}$, utilizing Allard's regularity
theorem \cite{Allard}, and present a rigorous derivation of the Euler-Lagrange equation satisfied by stationary points of \eqref{E2}. The additional regularity obtained therein allows us to make stronger statements concerning the limiting behavior of critical points in dimension $n=2$; in particular,
we show that, in the case of a bounded number of droplets which have bounded isoperimetric deficit, the generalized mean curvature of each connected component of $\{u^{\eps}=+1\}$ (appropriately normalized) is asymptotically constant.

Our paper is organized as follows. In Section \ref{sharpsetup} we set up certain notation which will be used throughout the paper, and present our three main results in Sections
\ref{sharpsection}, \ref{cpregularity} and \ref{diffusesection} respectively. In Section \ref{finper} we provide a brief introduction to the theory of sets of finite perimeter
and weak mean curvature. In Section \ref{sec4} we prove the main result of Section \ref{sharpsection} for
stationary points of the sharp interface energy \eqref{E2}. We then address the case of the diffuse interface energy (cf. equation \eqref{EE}) in Section \ref{sec5}, where we prove the main result of Section \ref{diffusesection}.\\

\textbf{Notation:}
 We will denote $\mathcal{D}'(\Omega)$ as the space of distributions on $\Omega$ and $H^k(\Omega)$ and $W^{k,p}(\Omega)$ will, as usual, denote the standard Sobolev spaces. We denote as $\mathcal{H}^{k}$ the standard
$k$-dimensional Hausdorff measure. For a measurable set $E \subset \Omega$, $P(E \cap \Omega)$ will denote its relative perimeter (see Section \ref{finper} for definitions), and $|E|$ will denote its standard n-dimensional Lebesgue measure. We write as $\mathbb{T}^n = [0,1)^n$, the standard flat n-dimensional torus.
With some abuse of notation, we will sometimes say $E \subset \mathcal{A}$ (or $\mathcal{A}_m$) when we mean $\chi_{E}$, the indicator function of $E$, belongs to $\mathcal{A}$ (respectively $\mathcal{A}_m$). Finally we denote $\alpha_{n-1}$ as the volume
of the unit ball in $\mathbb{R}^{n-1}$. 

\section{Problem formulation and main results}\label{sharpsetup}

In this section, we first rewrite the energy \eqref{E2} in a way which is more convenient for the subsequent presentation and analysis. We begin with the result of Ambrosio et al. \cite{ambrosio2} which allows us to decompose (up to $\mathcal{H}^{n-1}$ negligible sets) $\{u=+1\}$ into a countable collection of connected components $\{\Omega_i\}$ contained in a single cell of $\mathbb{T}^n$ when $\mathcal{H}^{n-1}(\{u=+1\})$ is sufficiently small:
\begin{equation}\label{F1}
 u(x) = -1 + 2\sum_{i} \chi_{\Omega_i}(x),
\end{equation}
and we set
\begin{equation}\label{F2}
 \bar u = -1 + \delta(\eps),
\end{equation}
where we assume $\delta(\eps)$ is bounded as $\eps \to 0$. We define the ``normalized droplet density"
\[\omega_{\eps} := \frac{\sum_i \chi_{\Omega_i}}{\sum_i |\Omega_i|},\]
so that $\omega_{\eps}$ is a probability measure on $\mathbb{T}^n$ for all $\eps>0$. If we insert \eqref{F1} and \eqref{F2} into the sharp interface energy \eqref{E2} we obtain
\begin{multline}\label{E2a}
E^{\eps}[u^{\eps}]  = \eps \sum_i P(\Omega_i) - \frac{2 \delta(\eps)}{\kappa^2}\sum_i |\Omega_i|  + 2\iint_{\mathbb{T}^n \times \mathbb{T}^n} G(x-y) \sum_i \chi_{\Omega_i}(x) \sum_i \chi_{\Omega_i}(y)dx\;dy\\ + \frac{\delta(\eps)^2}{2 \kappa^2},  
\end{multline}
where we set
\begin{align}\label{ubeps}
 -\Delta v_{\eps} + \kappa^2 v_{\eps} =  \frac{\sum_i \chi_{\Omega_i}}{\sum_i |\Omega_i|} =: \omega_{\eps}.
\end{align}
The rewriting of \eqref{E2} expressed by \eqref{E2a} will turn out to be more convenient for our purposes, as it allows us to focus on a non-local energy which depends only
on the normalized droplet density $\omega_{\eps}$  (and not $\bar u^{\eps}$). Our goal is to derive a suitably weak form of \eqref{cpeqn}. We proceed by computing the Euler-Lagrange equation of \eqref{E2} and show that this is equivalent to a certain $2$ tensor
$\{S_{ij}\}=S^{\eps}$ having zero divergence. The idea is then to pass to the limit in the condition
\[ \div S^{\eps} = 0\]
as $\eps \to 0$, and obtain a weak form of \eqref{cpeqn} as the limiting condition. This may at first appear surprising, as there will be contributions (in the form of curvature) from the perimeter term in \eqref{E2a}, and \eqref{cpeqn} seems to depend only on the non-local terms. As alluded to before, we show that the contributions from these local terms occur in a very small set so that we are still able to conclude \eqref{cpeqn} in an appropriately weak sense outside of this set, and this turns out to be enough to make our main conclusions. More precisely, we show that the set where the local terms are concentrated in the Euler-Lagrange equations have arbitrarily small 1-capacity. \\

\noindent We recall from Evans-Gariepy \cite{Gariepy} the definition of $p$-capacity of a set $E \subset \mathbb{R}^n$:
\[ \textrm{ Cap}_p(E) = \inf \left\{ \int_{\mathbb{R}^n} |\nabla \varphi|^p; \varphi \in L^{p^*}(\mathbb{R}^n), \nabla \varphi \in L^p(\mathbb{R}^n), E \subset \textrm{ int} (\varphi \geq 1)\right\},\]
where int$(A)$ denotes the interior of $A$ and $p^*=2p/(2-p)$. 
We will show that up to a set of very small 1-capacity, the tensor $S^{\eps}$ is close to the tensor $T^{\eps}$ in $L^1(\mathbb{T}^n)$ defined by
\[ T_{ij}^{\eps} = - \partial_i v_{\eps} \partial_j v_{\eps} + \frac{1}{2} \delta_{ij} (|\nabla v_{\eps}|^2 + \kappa^2 v_{\eps}^2),\]
where the condition
\[ \div T^{\eps} = 0\]
implies that 
\begin{equation}\label{cp1}  \omega_{\eps} \nabla v_{\eps} = 0 \textrm{ in } L_{loc}^1.\end{equation}

\noindent Our goal is to pass to the limit in this condition and obtain the weak form of \eqref{cpeqn}:
\begin{equation}\label{cp2}  \div T = 0,\end{equation}
up to a set of arbitrarily small 1-capacity, where $T$ is the 2-tensor with components $T_{ij}$ given by
\begin{equation}\label{Tdef} T_{ij} = - \partial_i v \partial_j v + \frac{1}{2} \delta_{ij} (|\nabla v|^2 + \kappa^2 v^2),\end{equation}
and $v$ is the distributional limit of $v^{\eps}$ (cf. equation \eqref{ubeps}) obtained from the weak convergence of $\omega_{\eps}$ to $\omega$. 
The condition \eqref{cp1} is in fact obtained by taking variations of the non-local term in $\eqref{E2a}$ of the form $v_t(x)=v(x + tX(x))$, often
called ``inner variations''. More precisely, condition \eqref{cp1} arises from the vanishing of
\[\frac{d}{dt}\Big|_{t=0} \int_{\mathbb{T}^n} |\nabla v_t|^2 + \kappa^2 v_t^2 dx.\]

The vanishing
of the divergence of this tensor (cf. equation \eqref{cp2} implies, in particular, that $v$ is constant on the support of $\omega$ if $\omega \in L^p(\mathbb{T}^n)$ for large
enough $p$ and a `vanishing gradient property', first established in \cite{BBH} in the context of Ginzburg-Landau, if $\omega = \sum_{i=1}^d b_i \delta_{a_i}$ (see Theorem 1), which formally states that the force
on each particle is balanced by the others.
 We now make some of these notions precise in order to state our main result, and begin with the following definition, taken from \cite{SS3}.

\begin{definition}\emph{\textbf{(Divergence-free in finite part)}} \label{divfreefin} Assume $X$ is a vector field in $\mathbb{T}^n$. We say $X$ is divergence-free in finite part if there exists
 a family of sets $\{E_{\delta}\}_{\delta > 0}$ such that
\begin{itemize}
 \item[1.] We have $\lim_{\delta \to 0} \textrm{\em Cap}_1(E_{\delta}) = 0$.
\item[2.] For every $\delta > 0$, $X \in L^1(\mathbb{T}^n \backslash E_{\delta})$.
\item[3.] For every $\zeta \in C^{\infty}(\mathbb{T}^n)$,
\[ \int_{\mathbb{T}^n \backslash F_{\delta}} X \cdot \nabla \zeta = 0,\]
where $F_{\delta}=\zeta^{-1} (\zeta(E_{\delta}))$. 
\end{itemize}
If $T$ is a 2-tensor with coefficients $\{T_{ij}\}_{1 \leq i,j \leq n}$ we say $T$ is divergence-free in finite part
if the vectors $T_i = (T_{i1},T_{i2}, \cdots, T_{in})$ are, for $i=1,2,\cdots,n$. 
\end{definition}

To see that the above definition is consistent with the ordinary notion of divergence free, we borrow the following proposition from \cite{SS3}.

\begin{proposition}\label{divfreenice}
 Assume that $X$ is divergence free in finite part in $\mathbb{T}^n$ and that $X \in L^1(\mathbb{T}^n \backslash E)$. Then for every $\zeta \in C_c^{\infty}(\mathbb{T}^n)$
we have
\[ \int_{\mathbb{T}^n \backslash F} X \cdot \nabla \zeta = 0,\]
where $F=\zeta^{-1}(\zeta(E))$. In particular, if $X$ is in $L^1(\mathbb{T}^n)$, then $F= \emptyset$ in the above and therefore $\div X = 0$ in $\mathcal{D}'(\mathbb{T}^n)$.
\end{proposition}

\subsection{Main result I: The sharp interface energy \eqref{E2}}\label{sharpsection}
Our first main result concerning stationary points of \eqref{E2} is the following.
\begin{theorem} \textbf{ \em (Equidistribution of droplets)} \label{main2}
 Let $u^{\eps} \in \mathcal{A}$ be a sequence of stationary points of \eqref{E2} in $\mathcal{A}$ in the sense of Definition 1 and assume %$\delta(\eps)=o_{\eps}(\sqrt{\eps})$ as $\eps \to 0$ and
\begin{align}\label{apriori1}
 %\limsup_{\eps \to 0} \delta(\eps)^{-2} E^{\eps}[u^{\eps}] +   &< +\infty.
 \limsup_{\eps \to 0} \int_{\mathbb{T}^n} |\nabla u^{\eps}| = 0.
\end{align}
Then for any $p \in (1,n/(n-1))$,  $\omega_{\eps}$ converges in $W^{-1,p}$ to a probability measure $\omega$ and
 $v_{\eps}$ converges in $W^{1,p}$ to $v$, where $v$ and $\omega$ are related via
\begin{equation}\label{vdef3}-\Delta v + \kappa^2 v = \omega.\end{equation}
 Moreover, the symmetric 2-tensor $T_{\omega}$ with coefficients
 $T_{ij}$ given by \eqref{Tdef} is divergence free in finite part. In addition, we have the following characterizations of the divergence free condition on $T_{\omega}$. 
\begin{itemize}
\item[0.] If $\int d\omega_{\eps} =0$ for all $\eps>0$ sufficiently small, then
\begin{equation} \omega \equiv 0.\end{equation}
 \item[1.] If $\omega \in H^{-1}(\mathbb{T}^n)$ then 
\begin{equation}\label{divfreecon} \div T = 0 \textrm{ in } \mathcal{D}'(\mathbb{T}^n).\end{equation}

\item [2.] If $\omega \in L^p$ for $p >1$ when $n=2$ and $p \geq 2n/(n+1)$ otherwise and $\omega \neq 0$, then in fact \[ \omega = \mathbf{1}\;dx,\] the uniform Lebesgue measure on $\mathbb{T}^n$.
\item[3.] If $\omega = \sum_{i=1}^d b_i \delta_{a_i}$ then setting $v(x) =  \Phi (|x-a_i|) +  H_i(x)$ where $\Phi$ is the fundamental
solution to the Laplace equation in $\mathbb{R}^n$ and $H_i$ is smooth in a neighborhood
of $a_i$, we have
\begin{equation}\label{VGP} \nabla H_i(a_i) = 0,\end{equation}
for $i=1,\cdots,d$. 
 \end{itemize}
\end{theorem}
Theorem \ref{main2} is analogous to the results obtained for Ginzburg-Landau \cite{SS5,SS3}, with the droplets playing the role of the vortices in the magnetic
Ginzburg-Landau model. The main difference in our case is that we are dealing with sharp interface version of \eqref{EE} so that $u^{\eps}$ takes on only the values
$+1$ and $-1$. We must therefore be careful concerning regularity issues on the boundary of the set $\{u^{\eps}=+1\}$, and consequently use the theory of finite perimeter sets (Section \ref{finper}). Our proof, however, is in some ways simpler as we will have no contributions from the local terms outside the support of $\omega_{\eps}$. This is no longer true for \eqref{EE} in Section \ref{sec5}, and some additional analysis is needed. In addition, the vortices in the Ginzburg-Landau model are quantized, and we do not a priori know the shape or volume of the droplets in
this model. Theorem \ref{rounddroplets} in Section \ref{cpregularity} provides some information about the shape of these droplets; in particular, they are asymptotically round
as $\eps \to 0$ when $n=2$ under assumptions on the number of droplets and their isoperimetric deficit ratio. We will see later that this is easily seen to be
false for dimensions $n \geq 3$. 

\subsection{Interpretation of Theorem \ref{main2}}
The hypothesis \eqref{apriori1} is essential to our proofs, as it will be seen to imply that $\textrm{Cap}_1(\{u^{\eps}=+1\}) = o_{\eps}(1)$ as $\eps \to 0$. This allows
us to show that $\div T^{\eps}$ converges, in a distributional sense, outside of the set $\{u^{\eps}=+1\}$ to $\div T_{\omega}$. The smallness of the set $\{u^{\eps}=+1\}$ allows us to demonstrate
that the limiting tensor $T_{\omega}$ is divergence free in finite parts. 

The conditions of Cases 2 and 3
are simply consequences of the divergence free condition on $T_{\omega}$ (see Section \ref{sec4}). The condition \eqref{VGP} is called the `vanishing
gradient property', first established in the context of Ginzburg-Landau in \cite{BBH} where $\{(a_i, b_i)\}_i$ is a critical
point of the ``renormalized energy'' associated to the problem. The condition \eqref{VGP} can be interpreted as saying the sum of the Coulombic forces
from the neighboring droplets balance each other.

 When $\omega$ is regular enough (Case 2) and non-zero, then in fact it is equal to the uniform Lebesgue measure
on $\mathbb{T}^n$, meaning the droplets are uniformly distributed throughout the domain.
When we only know that $\omega \in H^{-1}(\mathbb{T}^n)$ as in Case 1 above, the measure
$\omega$ can be concentrated on lower dimensional hypersurfaces \cite{Aydi,NamLe,SS5}. This concentration phenomenon also occurs in the two-dimensional magnetic Ginzburg-Landau model where the limiting vortices of solutions, which bear much resemblance with the droplets in our case, can concentrate on lines \cite{Aydi,SS3}. Analysis concerning the existence of solutions to \eqref{vdef3} in a bounded domain $\Omega$ with $\omega$ concentrated on a smooth, closed curve $\Sigma \subset \subset \Omega$, and absolutely continuous with respect to the arc-length measure on $\Sigma$, is studied in \cite{NamLe}. In all cases, the above analysis shows that we can have $\omega \in H^{-1}(\Omega)$, while it is not in general true that $\omega << dx$. Here we demonstrate a simple example on $\mathbb{T}^n$ for the screened Poisson kernel (an example for the non-screened kernel can be similarly constructed).

\begin{example}
Let $w(s)$ be the Green's function of the operator $-\Delta + \kappa^2 I$ on $[-1,1)$ with periodic boundary conditions. Then $w$ is the unique periodic solution to
\begin{equation}\label{1dpde}
-w''(s) + \kappa^2 w(s) = \delta(s) \textrm{ on } [-1,1),
\end{equation}
where $\delta(s)$ is the dirac delta function at $s=0$. Set $v(x_1,\cdots,x_n) = w(x_1)$ and we have
\[ -\Delta v(x_1) + \kappa^2 v(x_1) = \delta(x_1) \textrm{ on } [-1,1)^n.\]
In this case the divergence free condition $\div T_{\omega}=0$ is
equivalent to requiring that
\[ \int_{-1}^{1} (-v_{x_1}^2+\kappa^2 v^2) \phi'(s) ds = 0 \textrm{ for all } \phi \in C^1([-1,1)) \textrm{ periodic}.\]
It is easy to see from standard elliptic methods that $w \in W^{1,\infty}([-1,1])$ and $w \in C^{\infty}(B_{\rho}(0)^c)$ for any fixed $\rho >0$. Thus
\begin{align}\label{ex1a} \int_{-1}^{1} (-v_{x_1}^2+\kappa^2 v^2) \phi'(s) ds =  \int_{-1}^{-\rho}(-v_{x_1}^2+\kappa^2 v^2) \phi'(s) ds  + \int_{\rho}^1 (-v_{x_1}^2+\kappa^2 v^2) \phi'(s) ds + o_{\rho}(1),\end{align}
as $\rho \to 0$. Observing that if $w(s)$ solves \eqref{1dpde} then so does $w(-s)$, we conclude from uniqueness of solutions to \eqref{1dpde} that $w$ is an even function and therefore that $w'$ is odd. Thus integrating by parts and using periodicity of $v$ we obtain
\begin{multline}\label{ex1b}  \int_{-1}^{-\rho}(-v_{x_1}^2+\kappa^2 v^2) \phi'(s) ds  + \int_{\rho}^1 (-v_{x_1}^2+\kappa^2 v^2) \phi'(s) ds\\ =(-v_{x_1}^2+\kappa^2 v^2)\phi\Big|_{x_1=+\rho}-(-v_{x_1}^2+\kappa^2 v^2)\phi\Big|_{x_1=-\rho}=0.\end{multline}
Combing \eqref{ex1a} and \eqref{ex1b} and then sending $\rho \to 0$, we conclude $\div T_{\omega} = 0$ in $\mathcal{D}'([-1,1)^n)$.
\end{example}
\begin{comment}
\begin{example}
 Define the function
\begin{equation}\label{E.6}
v(x_1,\cdots,x_n)=\left\{\begin{array}{ccc}
x_1+1 &\mbox{for}&-1  \le x\le  -1/2\\
-x_1&\mbox{for}& -1/2  \le x\le 1/2\\
x_1-1 &\mbox{for}& 1/2 \le x \le 1
\end{array}.\right.
\end{equation}
Then $v$ is a distributional solution to
\[ - \Delta v (x_1,\cdots,x_n) = \omega - 1\;\;\;\; \textrm{    for }\;\;\;\; (x_1,\cdots,x_n) \in [-1,1)^n,\]
where
\[ \omega = -\delta(x_1-1/2) + \delta(x_1+1/2) + 1,\]
and $\delta(x_1)$ is the Dirac measure at $x_1=0$. The divergence free condition $\div T_{\omega}=0$ in this case is equivalent to requiring that
\[ \int_{-1}^{1} v_{x_1}^2(s)\phi'(s) ds = 0 \textrm{ for all } \phi \in C^{1}([-1,1)) \textrm{ periodic },\]
which is clearly satisfied since $|v_{x_1}| = 1$ a.e. In addition, the distribution $-\Delta v$ on $[-1,1)^n$ is in $H^{-1}([-1,1)^2)$ since a direct computation
yields
\[ \int_{[-1,1)^n} |\nabla v|^2 dx_1 \cdots dx_n < + \infty.\]

\end{example}

\end{comment}
In the following section we recall that we say $E \subset \mathcal{A}$ (respectively $E \subset \mathcal{A}_m$) if
the characteristic function of $E$, $\chi_{E}$, belongs to $\mathcal{A}$ (respectively $\mathcal{A}_m$). 

\subsection{Main Result II: Asymptotic roundness of droplets}\label{cpregularity}

We begin by recalling the main result of \cite{GV}, applied specifically to the torus. For $\gamma \in \mathbb R$ we consider the more general functional $I_\gamma: \mathcal A \to \mathbb R$ given by
\begin{equation}\label{nonloceqn1} I_\gamma(E): = P(E) + \gamma \int_E\int_E G(x,y)\,dy\,dx + \int_{E} f(x)dx,\end{equation}
where $f \in C^2(\mathbb{T}^n)$, $\gamma \in \mathbb{R}$ is a constant parameter, $P(E)$ is the perimeter of $E$ (see Section \ref{finper}) and $G \in L^1(\mathbb{T}^n \times \mathbb{T}^n)$ is the kernel of the Laplacian on $\mathbb{T}^n.$ %such that, setting $v_E(x) := (G(y-\cdot) * \chi_{E}(y))(x)$ we have
%\begin{equation}\label{estimate2}\|v_E\|_{W^{2,p}(K)} \leq C(n,K,E),\end{equation}
%for $p>n$ and each $E \subset \Omega$ and $K \subset \subset \Omega$. For instance, the Green's function of any uniformly elliptic, linear operator with Dirichlet, Neumann or periodic boundary conditions
%satisfies \eqref{estimate2} (see \cite{trudinger} for instance).

\noindent The reduced boundary of a set $E$ is said to be of class $C^{k,\alpha}$ if each point in $\partial^*E$ is locally contained in the graph of a function which is $C^{k,\alpha}$.
Our main result in \cite{GV} for the regularity of the reduced boundary is the following. 
  
\begin{theorem}\label{regularity}
 Let $E$ be a stationary point of the functional \eqref{nonloceqn1} in $\mathcal{A}$ or $\mathcal{A}_m$. Then the reduced boundary $\partial^* E$ belongs to the class $C^{3,1-n/p}$. In particular, the equation
\[ H(x) + 2 \gamma v_E + f(x) = \lambda,\]
holds strongly on $\partial^* E$ where $H$ is the mean curvature of $\partial^* E$, and
$\lambda$ is a Lagrange multiplier. When $E$ is a stationary point in the class $\mathcal{A}$, then $\lambda = 0$. Moreover, $\mathcal{H}^{n-1}(\partial E \backslash \partial^* E) = 0$.
\end{theorem}

 The proof of Theorem \ref{regularity} follows essentially from Allard's regularity theorem and De Giorgi's structure theorem. Theorem \ref{regularity} applied to \eqref{E2a} with $n=2$, $E=\{u^{\eps}=+1\}$ and $f=-\frac{2\delta(\eps)}{\kappa^2}$ yields the equation
\begin{equation}\label{ELeqn2}
 \eps H_{\eps} - \frac{2 \delta(\eps)}{\kappa^2} + v_{\eps} \sum_j |\Omega_j|  = 0 \textrm{ on } \partial^*\{u^{\eps}=+1\}.
\end{equation}
We will use \eqref{ELeqn2} to show that when the number of droplets is finite and they have bounded isoperimetric deficit, they become asymptotically round as $\eps \to 0$ in $n=2$.\\ 

\noindent We recall that the Green's function on $\mathbb{T}^2$ can be written as
\begin{equation}\label{greenlog}
 G(x-y) = -\frac{1}{2\pi} \log|x-y| + S(x-y) \textrm{ for } x,y \in \mathbb{T}^2,
\end{equation}
where $S$ is a continuous function. If we consider a single round droplet so that $u^{\eps} = -1 + \chi_{B(x,r_{\eps})}$, then formally we expect from \eqref{ELeqn2} and \eqref{greenlog} that
\begin{equation}
 H_{\eps} \simeq \frac{\log r_{\eps}}{\eps}r_{\eps}^2 + \eps^{-1}\delta(\eps).
\end{equation}
When $\eps^{-1}\delta(\eps) = O(r_{\eps}^2 \log r_{\eps})$, as is the case for minimizers \cite{m:cmp10,gms10b,gms11b}, then we have 
\begin{equation}\label{curvest} H_{\eps} = O\(\frac{\log r_{\eps}}{\eps}r_{\eps}^2\) \textrm{ as } \eps \to 0.\end{equation}
Equation \eqref{curvest} provides us with a hint of what the correct scaling of $H_{\eps}$ should be as the droplets shrink to points.\\

 We now make the assumption that $u^{\eps} = -1 + \sum_{j=1}^{N(\eps)}\chi_{\Omega_j}$ for $N(\eps)=O(1)$ as $\eps \to 0$ so that the number of droplets
is constrained to be finite. In the case that $u^{\eps}$ is minimizing, it is shown in \cite{m:cmp10,gms10b,gms11b} that any two droplets stay sufficiently far apart, and this is due to the the Coulombic repulsion between droplets arising in the non-local term when bounds on the energy are assumed. This is no longer
true in our case, and we must account for the situation where multiple droplets converge to the same point in $\mathbb{T}^2$, while still finding an appropriate normalization of $H_{\eps}$ as the droplets shrink to points.  Motivated from the above discussion, we define 
\begin{equation}\label{genrad}
 \rho_{\eps} :=  \frac{-\eps}{\sum_{j=1}^{N(\eps)} \log P( \Omega_{j_i}) \sum_{j=1}^{N(\eps)} |\Omega_j|},
\end{equation}
to be the ``normalized radius" and 
\begin{equation}\label{bardelta}
\bar \delta := \liminf_{\eps \to 0} \frac{ -\delta(\eps)}{\sum_{j=1}^{N(\eps)} \log P(\Omega_{j_i})\sum_{j=1}^{N(\eps)} |\Omega_j|},
\end{equation}
 to be the ``normalized volume fraction". When we work in the scaling regime of minimizers as in \cite{m:cmp10,gms10b,gms11b} then it is shown that there exists a $\bar \delta_{cr} > 0$ such that whenever $\bar \delta > \bar \delta_{cr}$ we have
$P(\Omega_j) = O(\eps^{1/3}|\ln \eps|^{-1/3})$, $|\Omega_j| = O(\eps^{2/3}|\ln \eps|^{-2/3}|)$ and thus, when $N_{\eps}=O(1)$ as $\eps \to 0$,
\[ \rho_{\eps} = O(\eps^{1/3} |\ln \eps|^{-1/3}) = O\(r_{\eps}\) \textrm{ as } \eps \to 0,\]
where $r_{\eps}=3^{1/3}\eps^{1/3}|\ln \eps|^{-1/3}$ is the energetically preferred radius of a single droplet as shown in \cite{m:cmp10, gms10b,gms11b}. We have the following
Theorem concerning the asymptotic roundness of droplets when $n=2$ as $\eps \to 0$. 

\begin{theorem}\emph{\textbf{(Asymptotic roundness of droplets when $n=2$)}}\label{rounddroplets}
Assume the hypotheses of Theorem \ref{main2} and, in addition, that $u^{\eps} = -1 + 2\sum_{i=1}^{N(\eps)} \chi_{\Omega_i}$ for $N(\eps)=O(1)$ as $\eps \to 0$ with
bounded isoperimetric deficit:
\begin{equation}\label{isodef}
 \limsup_{\eps \to 0} \frac{\sum_{j=1}^{N(\eps)} \textrm{\em P}(\Omega_j)^2}{\sum_{j=1}^{N(\eps)}|\Omega_j|} < +\infty,
\end{equation}
and $\bar \delta \in (0,+\infty)$. Then there exists a $\bar \delta_{cr}$ such that for $\bar \delta > \bar \delta_{cr}$ the following holds.
 Let $\Omega_{j_i}$  have center of mass converging (subsequentially) to $a_i$ for $j_i=1,\cdots,d_i$. Then there exists a constant $c_i > 0$ such that
such that
\begin{align}
 \left \|\rho_{\eps} H_{\eps} -c_i\right\|_{L^{\infty}\(\cup_{j_i=1}^{d_i} \partial \Omega_{j_i}\)} \to 0 \textrm{ as } \eps \to 0,\end{align}
up to subsequences, where $H_{\eps}$ is the mean curvature of $\{u^{\eps}=+1\}$ and $\rho_{\eps} $ is given by \eqref{genrad}. 
\end{theorem}
\begin{remark}
 The assumption \eqref{isodef} is required in order to ensure the next order
term in the expansion of the potential $v_{\eps}$ is controlled. In the case of minimizers as in \cite{m:cmp10,gms10b,gms11b}, bounds on the energy imply the condition
\eqref{isodef}. It is easy to see in dimensions $n \geq 3$ that the above statement is false, by taking any solution in $n=2$ and extending uniformly in the third direction we also
obtain a solution which is composed of tubes (and not spherical droplets). The proof works in dimension $n=2$ due to the specific scaling of the logarithmic potential, as can be seen by \eqref{greenlog} . Indeed, for very small droplets, the leading
order contribution from the potential $v_{\eps}$ is independent of the shape of the droplet. 
\end{remark}

\subsection{Main result III: The diffuse interface energy equation \eqref{EE}}\label{diffusesection}
For the diffuse interface energy \eqref{EE}, the analysis is very similar to that of the sharp interface energy \eqref{E2}, however we must use the unscreened kernel for the Laplace operator and thus define
\[ \tilde T_{ij} = - \partial_i v_{\eps} \partial_j v_{\eps} + \frac{1}{2} \delta_{ij} |\nabla v_{\eps}|^2,\]
where
\[v_{\eps}(x) = \int_{\mathbb{T}^n} G(x-y) \frac{1 + u^{\eps}(y)}{\delta(\eps)}dy,\]
and we make the particular choice $V(u) = \frac{1}{4}(1-u^2)^2$. 
%\[ -\Delta \tilde v^{\eps} = \frac{u^{\eps} - \bar u^{\eps}}{\delta(\eps)} \textrm{ on } \mathbb{T}^n,\]
We must now work in the class $\mathcal{A}_{\bar u}$ given by
\[ \mathcal{A}_{\bar u} := \left\{ u \in H^1(\mathbb{T}^n) : \int_{\mathbb{T}^n} u = \bar u\right\},\]
due to \eqref{neutr}.
%We thus consider variations $\phi_t(x) := x + tX(x)$ which, in addition, satisfy $ \in \mathcal{A}_{\bar u}$ for small $|t|$. 
For the energy \eqref{EE}, we define a critical point as follows.
\begin{definition}\label{diffusestationary}
A function $u \in \mathcal{A}_{\bar u}$ is said to be a critical point of \eqref{EE} if for any $v \in H^1(\mathbb{T}^n)$ satisfying
$\int_{\mathbb{T}^n} v = 0$ we have
\[\frac{d}{dt}\Big|_{t=0} \mathcal{E}^{\eps}(u + t v) = 0.\]
\end{definition}
\noindent A simple calculation along with standard elliptic theory reveals that $u^{\eps}$ is $C^{3,\alpha}$ and solves the elliptic equation
\begin{equation}
 -\eps^2 \Delta u^{\eps} + u^{\eps}(1-(u^{\eps})^2) + \delta(\eps) v_{\eps} = \lambda_{\eps} \textrm{ in } \mathbb{T}^n, 
\end{equation}
where $\lambda_{\eps}$ is the Lagrange multiplier corresponding to the volume constraint when taking variations in Definition \ref{diffusestationary}.

We show that if $u^{\eps} \in \mathcal{A}_{\bar u^{\eps}}$, with $\bar u^{\eps} = -1 + \delta(\eps)$, is a sequence of critical points of $\mathcal{E}^{\eps}$ with the perimeter of the minority phase vanishing, then $\tilde T^{\eps}$ converges up to a small set to the tensor $\tilde T_{\omega}$ with coefficients defined by
\begin{equation} \label{difften}\tilde T_{ij} = - \partial_i v\partial_j v + \frac{1}{2} \delta_{ij} |\nabla v|^2,\end{equation}
where now
\[-\Delta v = \omega - 1\textrm{ on } \mathbb{T}^n,\]
and $\omega$ is a probability measure on $\mathbb{T}^n$. More precisely, we prove the following.
\begin{theorem}\emph{\textbf{(Diffuse interface energy)}}\label{main3}
 Let $u^{\eps} \in \mathcal{A}_{\bar u^{\eps}}$ be a sequence of critical points of \eqref{EE} in the sense of Definition \ref{diffusestationary} which satisfy $\limsup_{\eps} |\lambda_{\eps}| < +\infty$ and
\begin{equation}\label{aprioriD}
% \limsup_{\eps \to 0} \frac{1}{\delta(\eps)^2} \int_{\mathbb{T}^n} \eps^2 |\nabla u^{\eps}|^2 + \frac{1}{2}((u^{\eps})^2-1)^2 dx + |\lambda_{\eps}| < +\infty
\limsup_{\eps \to 0} \mathcal{H}^{n-1}\( \{u^{\eps} \geq -1 + \delta(\eps)^{1+\alpha}\}\) = 0 \textrm{ for } \alpha > 0,
\end{equation}
 with \begin{equation} \bar u^{\eps} = -1 + \delta(\eps) \textrm{ and } \delta(\eps) = o_{\eps}(1) \textrm{ as } \eps \to 0.\end{equation} Then for any $p \in (1,n/(n-1))$,  $ \omega_{\eps} := \frac{1+u^{\eps}}{\delta(\eps)}$ converges in $W^{-1,p}$ to a probability measure $\omega$ and
 $v_{\eps}$ converges in $W^{1,p}$ to $v$ where
\[-\Delta v = \omega - 1\textrm{ on } \mathbb{T}^n.\]
 Moreover, the symmetric 2-tensor $T_{\omega}$ with coefficients
 $T_{ij}$ given by \eqref{difften} is divergence free in finite part. In particular, cases 0., 1., 2. and 3. of Theorem \ref{main2} continue to hold for $\omega$.
\end{theorem}
\begin{remark}
The specific choice of $\delta(\eps)^{1+\alpha}$ in \eqref{aprioriD} is a technical limitation which is required in the proofs.\end{remark}

\section{Mathematical preliminaries: Sets of finite perimeter}\label{finper}

Here we introduce the basic notions of sets of finite perimeter. A detailed exposition on these topics can be found in \cite{maggi}. For a more
general treatment of varifolds, we refer the reader to \cite{Simons}.  Let $E \in \mathbb{R}^n$ be a Lebesgue measurable set. We say that
$E$ has finite perimeter if
\begin{equation}
\sup_{\substack{\varphi \in C_c^1(\mathbb{R}^n)\\ \|\varphi\|_{L^{\infty}} \leq 1}} \int_{E} \div \varphi  < +\infty.
\end{equation}
By the Riesz-Representation theorem, the above implies the existence of a vector valued Radon measure $\mu_E$ such that generalized Gauss-Green
formula  holds true
\[ \int_E \nabla \varphi = \int_{\mathbb{R}^n} \varphi d\mu_E \textrm{ for all } \varphi \in C_c^1(\mathbb{R}^n).\]
The measure $\mu_E$ is referred to as the Gauss-Green measure of $E$ and the total perimeter of the set $E$ is defined as
\[ P(E) = |\mu_E|(\mathbb{R}^n).\]

\noindent In the case that $E$ has a $C^1$ boundary, then we have
\begin{align*}
 \mu_E &= \nu_E \mathcal{H}^{n-1} \llcorner \partial E\\
P(E) &= \mathcal{H}^{n-1}(\partial E),
\end{align*}
and, in particular, we have
\[ \nu_E(x) = \lim_{r \to 0^+} \dashint_{B(x,r) \cap \partial E} \nu_E d\mathcal{H}^{n-1} = \lim_{r \to 0^+} \frac{\mu_E(B(x,r))}{|\mu_E|(B(x,r)}.\]

For a generic set $E$ of finite perimeter, we therefore define the \emph{reduced boundary}, denoted $\partial^*E$, as those $x \in \partial E$ such that the above limit exists and belongs to $S^{n-1}$.
The Borel vector field $\nu_E: \partial^*E \to S^{n-1}$ is called the \emph{measure theoretic unit normal of $E$}. When $\partial E$ is $C^1$, then
the measure-theoretic outer unit normal agrees with the classical definition.

\subsection{The first variation of perimeter}
We wish to define a one-parameter family of diffeomorphisms with initial velocity $X \in C_c^{1}(\Omega; \mathbb{R}^n)$
which is a collection $\{\phi_t\}_{t \in (-\varepsilon,\varepsilon)}$for $\eps > 0$ defined as 
\begin{equation}
 \phi_t(x) = x + t X(x), \textrm{ } x \in \Omega.
\end{equation}
\begin{comment}
the solution to the ODE
\begin{align}
 \frac{\partial}{\partial t} \phi(t,x) &= X(\phi(t,x)) \textrm{ for } x \in \mathbb{R}^n\\
\phi(0,x) &= x \;\;\;\;\;\;\;\; \;\;\;\;\;\;\textrm{ for } x \in \mathbb{R}^n,
\end{align}
\end{comment}
We call $\{\phi_t\}_{-\eps < t < \eps}$ a \emph{local variation in} $\Omega$ associated with $X$ if in addition
\begin{equation}
 \phi_t(\Omega) \subset \subset \Omega.
\end{equation}
The \emph{first variation of perimeter} is then easily computed as (see  \cite{maggi, Gariepy, Simons})
\begin{equation}\label{pervar} \frac{d}{dt}\Big|_{t=0} P(\phi_t(E)) = \int \textrm{div}_{E} X d\mathcal{H}^{n-1}, \;\;\; X \in C_0^1(\Omega;\mathbb{R}^n),\end{equation}
where $\textrm{div}_E X$ is the tangential divergence of the vector field $X$ with respect to $E$:
\[ \textrm{div}_E X = \textrm{div} X - \nu_E(x) \cdot \nabla X(x) \nu_E(x).\]
Observe that the first variation is a linear functional on $C_0^1(\Omega;\mathbb{R}^n)$. In the special
case that it has a continuous extension to $C_0^0(\Omega;\mathbb{R}^n)$ it can be represented by a vector
valued Radon measure, which has a singular part with respect to $\mu_E$ and a non-singular part, using
the Radon-Nikodym theorem.

\noindent We thus have
\begin{equation}\label{gencurv}
 \int \textrm{div}_E X d\mathcal{H}^{n-1} = - \int X \cdot \vec H d\mathcal{H}^{n-1}- \int X \cdot  \nu_E d\sigma_E,
\end{equation}
where $|\vec H| \in L_{loc}^p (\partial^* E)$ and $\sigma_E$ denotes the singular part of the measure.
We call $\vec H$ the \emph{vector valued generalized mean curvature}. When we can write $\vec H = H \nu_E$, we call
$H$ the \emph{generalized mean curvature}. 

\section{Proof of Theorem \ref{main2}}\label{sec4}

\begin{comment}
If we focus justcomputation yields the Euler-Lagrange equation $\div T^{\eps}=0$ where $T^{\eps}$ is as above. More precisely
we have the following proposition. 
\begin{proposition}\label{nonlocalvar} Let $v_{\omega} \in H^1(\mathbb{T}^d)$ solve $-\Delta v_{\omega} + \kappa^2 v_{\omega} = \omega$ in $\mathbb{T}^n$ where $\omega \in H^{-1}(\mathbb{T}^n)$. Then setting $v_t = v_{\omega}(x + tX(x))$ where $X$ is a $C^1$ vector field in $\mathbb{T}^n$, we have
\begin{equation}\label{innervar}
 \frac{d}{dt}\int_{\mathbb{T}^n} |\nabla v_t|^2 + \kappa^2 v_t^2 dx = \int_{\mathbb{T}^n}  \sum_{i,j=1}^n \partial_i X_j T_{ij}dx
\end{equation}
where $T_{ij}$ is the 2-tensor defied by
\begin{equation}
 T_{ij} = -\partial_i v \partial_j v + \frac{1}{2}\(|\nabla v|^2 + \kappa^2 v^2 \) \delta_{ij}.
\end{equation}

In particular if $\omega \in L^p(\mathbb{T}^n)$ for $p \geq 1$ we have
\begin{equation}
\frac{d}{dt}\int_{\mathbb{T}^2} |\nabla v_t|^2 + \kappa^2 v_t^2 dx = \int \nabla v_{\omega} \cdot X d\omega.
\end{equation}

\end{proposition}
\end{comment}
As seen previously in Section 2 (cf. equation \eqref{cp1}), a direct computation yields
\begin{equation}\label{divT}
 \div T^{\eps} = \nabla v_{\eps} \omega_{\eps} \textrm{ in } L_{loc}^1(\mathbb{T}^n),
\end{equation}
where $T^{\eps}$ is the 2-tensor with coefficients $T_{ij}$ given by
\begin{equation}
 T_{ij} = -\partial_i v_{\eps} \partial_j v_{\eps} + \frac{1}{2}\(|\nabla v_{\eps}|^2 + \kappa^2 v_{\eps}^2 \) \delta_{ij}.
\end{equation}

As discussed in the beginning of Section \ref{sharpsetup}, we proceed by showing that the Euler-Lagrange equation obtained in Theorem \ref{regularity} is equivalent to the vanishing of a certain 2-tensor $S_{\eps}$. The part of $S^{\eps}$ which does not include $T^{\eps}$ will be shown to be concentrated on $\{u^{\eps}=+1\}$, which will
be shown to have vanishing 1-capacity as $\eps \to 0$, as a result of our assumption that $\mathcal{H}^{n-1}(\{u^{\eps}=+1\})$ vanishes
as $\eps \to 0$. The first step is the following
proposition, which has been adapted from \cite{SS3} and generalized to dimensions $n\geq 2$. The purpose of it will become clear in the proof of
Theorem \ref{main2}, where we will cover the set $\{u^{\eps}=+1\}$ by small balls and use the fact that the 1-capacity of a ball $B(x,r)$
is $\alpha_{n-1} r^{n-1}$ \cite{Gariepy}.
\begin{comment}
\begin{proposition}(Evans-Gariepy) \label{capacityhaus}
 Let $A \subset \mathbb{R}^n$. There exists constants $C=C(p,n)$ and $c=c(p,n)$ depending only on $p$ and $n$ such that
\[ c|A|^{\frac{n-p}{n}} \leq \textrm{\em Cap}_p(A) \leq C \mathcal{H}^{n-p}(A).\]
\end{proposition}
\end{comment}

\begin{proposition}\label{sumballs}
 Assume $K$ is a compact subset of $\mathbb{R}^n$. Then there exists a finite covering of $K$ by closed balls $B_1,\cdots,B_k$ such that
\[\sum_k r(B_k)^{n-1} \leq C \mathcal{H}^{n-1}(\partial K).\]
\end{proposition}
\begin{proof}
 Since $\partial K$ is compact it suffices to work with a finite covering, and then taking closures and using Lemma 4.1 of \cite{SS3}, we may
assume the balls are closed and disjoint, by possibly increasing the constant $C$ in the proposition. Indeed if $B_1$ and $B_2$ are two
balls which intersect, then there exists a ball $B$ containing $B_1 \cup B_2$ such that $r(B)\leq r(B_1)+r(B_2)$ and thus $r(B)^{n-1} \leq C(r(B_1)^{n-1}+r(B_2)^{n-1})$.

 In particular $A=\mathbb{R}^n \backslash \bigcup_{i=1}^k B_i$
is connected. Now if $B_1,\cdots,B_k$ cover $\partial K$, we claim they cover $K$. The claim follows by noting that $A$, which is connected, intersects
the compliment of $K$ since $K$ is bounded. Thus if $A$ intersected $K$ it would also intersect $\partial K$, which is impossible from the definition
of $A$. Thus $K \subset \mathbb{R}^n \backslash A = \bigcup_{i=1}^k B_i$. The result then follows by the definition of $n-1$ dimensional Hausdorff measure.
\end{proof}
\begin{comment}
\begin{proposition}(Evans-Gariepy)
 Define the set $A = \{ x \in \mathbb{T}^n : u(x) \geq 1\}$ for $u \in BV(\mathbb{T}^n;\{-1,+1\})$. Then
\[ \textrm{\emph{Cap}}_1(A) \leq C \int_{\mathbb{T}^n} |\nabla u| dx,\]
where $C$ depends only on $n$.
\end{proposition}
\begin{proof}
 Let $A^o$ denote the interior of the set $A$. Then in $A^o$ we have $(u)_{x,r} > 1-\delta$ for $r$ sufficiently small and fixed
$\delta > 0$ and define
\[ A_{\delta} = \{ x \in \mathbb{T}^n : (u)_{x,r} > 1-\delta \textrm{ for some } r>0\}.\]
Then by \cite[Lemma 1 Section 4.8]{Gariepy} we have
\[ \textrm{Cap}_1(A_{\delta}) \leq \frac{C}{1-\delta} \int_{\mathbb{T}^n} |\nabla u| dx.\]
By vii) Theorem 2 in Section 4.7.1 of \cite{Gariepy}, $\textrm{Cap}_1$ is countably additive. Thus since $A^o = \bigcup_{k > 0} A_{2^{-k}}$ we have
\[ \textrm{Cap}_1(A^0) \leq C\int_{\mathbb{T}^n} |\nabla u| dx.\]
By \cite[Theorem 3, Section 4.7.2]{Gariepy} we have $\textrm{Cap}_1(\partial A) = 0$ since $\mathcal{H}^{n-1}(\partial A) < +\infty$. By subadditivity
of capacity (cf. \cite{Gariepy}) we have 
\[\textrm{Cap}_1(A)\leq \textrm{Cap}_1(A^0) + \textrm{Cap}_1(\partial A)=\textrm{Cap}_1(A^0),\]
from which the result follows.
\end{proof}
\end{comment}
We now finally define precisely what we mean by $L^1$ convergence `up to a small set'. This definition is taken from \cite{SS3}.
\begin{definition}\label{deltaconv}
 We say a sequence $\{X_k\}_k$ in $L^1(\Omega)$ converges in $L_{\delta}^1(\Omega)$ to $X$ if $X_k \to X$ in $L_{loc}^1(\Omega)$ except
on a set of arbitrarily small 1-capacity, or precisely if there exists a family of sets $\{E_{\delta}\}_{\delta>0}$ such that for any
compact $K \subset \Omega$,
\[ \lim_{\delta \to 0} \textrm{\em Cap}_1(K \cap E_{\delta}) = 0, \;\;\;\; \forall \delta > 0 \lim_{k \to +\infty} \int_{K \backslash E_{\delta}} |X_k-X| = 0.\]
We define similarly the convergence in $L_{\delta}^2$ by replacing $L^1$ by $L^2$ in the above.
\end{definition}

It is clear that $\nabla v^{\eps}$ cannot converge to $\nabla v$ strongly in $L^2$ in general, even if we have a uniform bound in $H^1(\mathbb{T}^n)$. 
However the fundamental observation is that away from a set of very small 1-capacity, we do in fact have strong $L^2$ convergence as long as the measures
 converge weakly in $(C(\mathbb{T}^n))^*$. The following result is adapted from \cite{SS3} to work in higher dimensions.
\begin{proposition}\label{msrconv}
 Assume $\{\alpha_k\}_k$ is a sequence of measures such that for some $p \in (1,n/(n-1))$
\[ \lim_{k \to +\infty} \|\alpha_k\|_{W^{-1,p}(\Omega)} \|\alpha_n\|_{C^0(\Omega)^*} = 0,\]
for $\Omega \subset \mathbb{R}^n$ bounded and open where $\|\alpha_k\|_{C^0(\Omega)^*}$ denotes the total variation of $\alpha_k$, $\int_{\Omega} |\alpha_k|$. Then letting
$h_k$ be the solution of
\[ -\Delta h_k + \kappa^2 h_k = \alpha_k  \mbox{ in } \Omega,\]
it holds that $h_k$ and $\nabla h_k$ converge to $0$ in $L_{\delta}^2(\Omega)$ . 
\end{proposition}
\begin{proof}
 We begin by noticing that $W^{1,q}$ embeds into $C^0$ for $q > n$, and thus the $(C^0)^*$ norm
dominates the $W^{-1,p}$ norm for $p \in (1,n/(n-1))$. Thus the hypothesis implies that $\|\alpha_k\|_{W^{-1,p}}$ tends to zero
as $k \to +\infty$. We let
\begin{align}\label{S1.1}
 \delta_k = \( \frac{\|\alpha_k\|_{W^{-1,p}}}{\|\alpha_k\|_{(C^0)^*}+1}\)^{1/2}, \;\;\; F_k = \{x \in \Omega | |h_k| \geq \delta_k\}.
\end{align}
Then we use the well known bound on p-capacity of $F_k$ (see \cite[Lemma 1]{Gariepy})
\begin{equation}\label{S1.2}
  \textrm{Cap}_p(F_k) \leq C \frac{\|h_k\|_{W^{1,p}}^p}{\delta_k^p}.
\end{equation}
Then by elliptic regularity we have $\|h_k\|_{W^{1,p}} \leq C \|\alpha_k\|_{W^{-1,p}}$ and so from \eqref{S1.1}--\eqref{S1.2} we have
\[ \textrm{Cap}_p(F_k) \leq C \|\alpha_k\|_{W^{-1,p}}^{p/2}(\|\alpha_k\|_{C^0(\Omega)^*} + 1)^{p/2},
\]
which therefore tends to $0$ as $k \to +\infty$. This implies that $\textrm{Cap}_1(F_k) \to 0$ as $n \to +\infty$. From a well known
property of Sobolev functions, the truncated function $\bar h_k = \max(-\delta_k, \min(h_k,\delta_k))$ satisfies $\nabla \bar h_k = 0$ a.e
in $F_k$, hence
\[ \int_{\Omega \backslash F_k} |\nabla h_k|^2 = \int_{\Omega} \nabla h_k \cdot \nabla \bar h_k.\]
It follows that
\[\int_{\Omega \backslash F_k} |\nabla h_k|^2 + \kappa^2h_k^2 \leq \int_{\Omega} \nabla h_k \cdot \nabla \bar h_k + \kappa^2 h_k \bar h_k = \int_{\Omega} \bar h_k d\alpha_k,\]
where the last equality follows from $-\Delta h_k + \kappa^2 h_k = \alpha_k$. The right hand side is bounded above by $\delta_k \|\alpha_k\|_{C^0(\Omega)^*}$, hence
by $\(|\alpha_k\|_{W^{-1,p}}\|\alpha_k\|_{C^0(\Omega)^*}\)^{1/2}$ and therefore tends to zero as $k \to +\infty$. Thus
\begin{equation}\label{l2}\lim_{k \to +\infty} \|h_k\|_{L^2(\Omega \backslash F_k)} = \lim_{k \to +\infty} \|\nabla h_k\|_{L^2(\Omega \backslash F_k)} = 0.\end{equation}
To conclude, since $\lim_{k \to +\infty} \textrm{Cap}_1(F_k) = 0$ there is a subsequence still denoted by $\{k\}$ so that
$\sum_k \textrm{Cap}_1(F_k) < +\infty$. We define
\[E_{\delta} = \bigcup_{k > \frac{1}{\delta}} F_k.\]
Then $\textrm{Cap}_1(E_{\delta})$ tends to zero as $\delta \to 0$ since it is bounded above by the tail of a convergent series. Moreover, for any
$\delta > 0$ we have $F_k \subset E_{\delta}$ when $k$ is large enough and therefore \eqref{l2} implies that $\lim_{k \to +\infty} \|h_k\|_{L^2(\Omega \backslash E_{\delta})} = \lim_{k \to +\infty} \|\nabla h_k\|_{L^2(\Omega \backslash E_{\delta})} = 0$.
\end{proof}

We will see in the proof of Theorem \ref{main2} that Proposition \ref{msrconv} implies that $T^{\eps}$ converges to $T$ in $L_{\delta}^1(\mathbb{T}^n)$. The proof of Theorem \ref{main2} then follows after applying the following proposition contained in \cite{SS3}.

\begin{proposition}\label{divfree}
 Assume $\{T_k\}_{k \in \mathbb{N}}$ is a sequence of divergence-free vector fields which converge to $T$ in $L_{\delta}^1(\mathbb{T}^n)$. Then $T$
is divergence-free in finite part. 
\end{proposition}

We are now ready to present the proof of Theorem \ref{main2}. The characterizations of $\omega$ in items 0,1,2,3 will be contained in Propositions \ref{omega1}
and \ref{omega2} below.\\
\medskip

\noindent \emph{Proof of first part of Theorem \ref{main2}:}
\begin{comment}
First we compute the first variation with respect to a vector field $X$: 
\begin{align}\label{firstvar3} \frac{d}{dt} E^{\eps}[u^{\epsilon}(x+tX(t))] = 0,%\frac{\eps}{\delta(\eps)^2} \sum_i \int_{\partial \Omega_i} H \cdot X dS\\ - \frac{\kappa^2}{\delta(\eps)} \sum_i \int_{\partial \Omega_i} X \cdot \nu dS +  \int_{\mathbb{T}^2} \mu_{\eps}\nabla v_{\eps}\cdot X_{\eps} = 0.
\end{align}
which from Proposition \ref{1st variation of perimeter} with \eqref{E2a} and Proposition \ref{firstvar1} is equivalent to 
\[ \div S^{\eps} = 0 \textrm{ in } \mathcal{D}'(\mathbb{T}^n),\]
where, letting $\mu_{\eps}$ be the Gauss-Green measure of $\{u^{\eps}=+1\}$ as defined in Section \ref{finper}, $S_{\eps}$ is the 2-tensor given by
\end{comment}
We begin by observing that if we define $J^{\eps}$ to be the 2-tensor with coefficients $J_{ij} = ( \delta_{ij} - \nu_i \nu_j)|\mu_{\eps}|$, where $\mu_{\eps}$ is the Gauss-Green
measure of $\{u^{\eps}=+1\}$ as in Section \ref{finper}, we have
\begin{equation}\int_{\partial \{u^{\eps}=+1\}} \div_E X d\mathcal{H}^{n-1}
 = \int_{\mathbb{T}^n} (\div_E X - \partial_iX^j\nu_i\nu_j )\,d|\mu_{\eps}|
 = \int_{\mathbb{T}^n} J_{ij} \partial_i X^j.\label{perfirstvar}\end{equation}

By Theorem \ref{regularity} applied to \eqref{E2a}, \eqref{divT} and \eqref{perfirstvar}, we claim the criticality condition for $E^{\eps}$ can be written as
\[ \div S^{\eps} = 0 \textrm{ in } \mathcal{D}'(\mathbb{T}^n),\]
where $S_{\eps}$ is the 2-tensor given by
\begin{equation}
 S_{ij}^{\eps} = T_{ij}^{\eps} - \frac{\eps}{b_{\eps}^2} \( \delta_{ij} - \nu_i \nu_j\)|\mu_{\eps}|  +  \frac{1}{b_{\eps}^2}\delta_{ij} \frac{(u_{\eps}+1)\delta(\eps)}{\kappa^2 } - \delta_{ij} v_{\eps} \omega_{\eps},
\end{equation}
where we've set $b_{\eps} = \sum_j |\Omega_j|$. Indeed, applying Theorem \ref{regularity} to \eqref{E2a} with $E=\{u^{\eps}=+1\}$, $f=-\frac{2\delta(\eps)}{\kappa^2}$,  $\Omega=\mathbb{T}^n$ with Green's potential $G$ of $\mathbb{T}^n$ we have

\[ \frac{\eps}{b_{\eps}^2} H_{\eps} \mu_{\eps} - \frac{\delta(\eps)}{b_{\eps}^2 \kappa^2} \mu_{\eps} + \frac{1}{b_{\eps}} v_{\eps}\mu_{\eps} = 0 \textrm{ in } \mathcal{D}'(\mathbb{T}^n).\]
Using \eqref{divT} and \eqref{perfirstvar}, a direct computation yields
\begin{equation}
 \div S^{\eps} = \nabla v_{\eps} \omega_{\eps} - \frac{\eps}{b_{\eps}^2} H_{\eps} \mu_{\eps} + \frac{\delta(\eps)}{b_{\eps}^2 \kappa^2} \mu_{\eps} - \nabla v_{\eps} \omega_{\eps} - \frac{1}{b_{\eps}} v_{\eps}\mu_{\eps}=0 \textrm{ in } \mathcal{D}'(\mathbb{T}^n).
\end{equation}
From Proposition \ref{sumballs}, there exists
a collection of balls $B_1,\cdots,B_k$ which cover $\{u^{\eps}=+1\}$ with $\sum_{i=1}^k r(B_i)^{n-1} \leq C \mathcal{H}^{n-1}(\{u^{\eps}=1\})$. Define $Z_{\eps}$ to be the union of these balls. Then we have
\[ S_{\eps} = T_{\eps} \textrm{ in } Z_{\eps}^c.\]
By subadditivity of the 1-capacity \cite{Gariepy} and the fact that the 1-capacity of a ball $B(x,r)$ is $\alpha_{n-1}r^{n-1}$ \cite{Gariepy} we have via the vanishing of $P(\{u^{\eps}=+1\})$ (cf. equation \eqref{apriori1}) that
\[ \textrm{Cap}_1 (Z_{\eps}) \to 0, \,\,\,\, \int_{\mathbb{T}^n \backslash Z_{\eps}} |S_{\eps} - T_{\eps}| = 0.\]

Now choose a decreasing subsequence $\{\eps_k\}$ tending to zero such that $\sum_k \textrm{Cap}_1(Z_{\eps_k}) < +\infty$ and let
\[ E_{\delta} = \bigcup_{k > \frac{1}{\delta}} Z_{\eps_k}.\]

\noindent Finally we define
\begin{equation}\label{deltadef1} F_{\delta} := E_{\delta} \cup \tilde E_{\delta},\end{equation}
where, in view of the defintion of $L_{\delta}^2$ convergence (cf. Definition \ref{deltaconv}), $\tilde E_{\delta}$ are the sets given by Proposition \ref{msrconv}. Then once again by subadditivity of capacity we have 
\[\lim_{\delta \to 0} \textrm{ Cap}_1(F_{\delta}) = 0.\]

Since $\omega_{\eps}$ is a family of probability measures on $\mathbb{T}^n$, we have $\omega_{\eps} \to \omega$ weakly in $(C^0(\mathbb{T}^n))^*$ up to a subsequence, and thus $\omega_{\eps} \to \omega$ strongly
in $W^{-1,p}$ for $p \in (1,n/(n-1))$ via the compact embedding $(C^0(\mathbb{T}^n))^* \subset \subset W^{-1,p}$ which follows from the
compact embedding $W^{1,q}(\mathbb{T}^n) \subset \subset C^0(\mathbb{T}^n)$ for $q > n$. From Proposition \ref{msrconv} we therefore conclude
\[ \nabla v_{\eps} \to \nabla v \textrm{ in } L_{\delta}^2(\mathbb{T}^n).\]

\noindent Thus, recalling $S_{\eps}=T_{\eps}$ in $F_{\delta}^c$ we have
\[ S_{\eps} - T_{\omega} \textrm{ converges to } 0 \textrm{ in } L_{\delta}^1(\mathbb{T}^n),\]
where the sets $F_{\delta}$ in Definition \ref{divfreefin} are given by \eqref{deltadef1}. Thus $T_{\omega}$ is divergence free in finite part from
Proposition \ref{divfree}. $\Box$

\noindent It now remains to prove the characterizations of Theorem \ref{main2}, ie. items 0, 1, 2 and 3, which we divide
into Propositions \ref{omega1} and \ref{omega2} below.

\begin{proposition}\label{omega1} Let $-\Delta v + v = \omega \in H^{-1}(\mathbb{T}^n)$ and that $T_{\omega}$ is divergence free in finite parts. Then it holds distributionally that
 \[ \div T_{\omega} = 0.\]
Moreover we have the following
\begin{itemize}
\item If $n=2$ then $v \in W^{1,\infty}$.
\item If, in addition, $\omega \in L^p$ for $p \geq \frac{2n}{n+1}$ when $n>2$ and $p \geq 1$ for $n=2$, then 
\[ \omega = \textbf{1}dx.\]
\end{itemize}
\end{proposition}

\begin{proof}
 When $n=2$ for both cases, see \cite{SS3}. The proofs are very similar to \cite{SS3} but we generalize them for arbitrary dimension. First observe that $\div T_{\omega}=0$ is an immediate consequence of Proposition \ref{divfreenice}. When $n>2$, if $\omega \in L^p$ for $p \geq \frac{2n}{n+1}$ then $\nabla v \in L^q$ for $q \leq \frac{p}{p-1}$
by standard elliptic theory. Let $\omega_n = \omega * \rho_n$  where $\{\rho_n\}_n$ is a regularizing kernel and define $v_k= v * \rho_k$ and let
$T_k$ be the tensor with coefficients $-\partial_i v_k \partial_j v_k + \frac{1}{2} (|\nabla v_k|^2 + v_k^2)\delta_{ij}$. Then $\omega_k$ tends
to $\omega$ in $L^p$ and since $\nabla v \in L^{p/(p-1)}$ , $\nabla v_n$ tends to $\nabla v$ in $L^{p/(p-1)}$. By H\"{o}lder's inequality we obtain
\[ \omega_k \nabla v_k \to \omega \nabla v, \;\;\; T_n \to T_{\omega} \textrm{ in } L_{loc}^1(\mathbb{T}^n).\]
It follows that $\div T_k \to \div T_{\omega} = 0$ and that $\omega_k \nabla v_k \to \omega \nabla v$ in $\mathcal{D}'(\mathbb{T}^n)$. Since 
$\div T_k = \omega_k \nabla v_k$ we conclude $\omega \nabla v = \lim_{k} \div T_k = 0$ in $L_{loc}^1(\mathbb{T}^n)$ and thus a.e. Then since
$\Delta v = 0$ a.e on the set $F = \{\nabla v = 0\}$, we have $\omega = \kappa^2 v$ a.e on the set $F$, and $\omega=0$ a.e on the complement of $F$ from $\omega \nabla v = 0$.
Thus we obtain \[ \omega= \kappa^2 v \mathbf{1}_{|\nabla v| = 0}.\]
Multiply by $v$ and integrating by parts, using the periodic boundary conditions on $\mathbb{T}^n$ we obtain that $\nabla v = 0$ a.e and thus
$v$, and therefore $\omega$ is constant. Since $\int \omega = 1$ it follows that $\omega = \textbf{1}dx$. 
\end{proof}

\begin{comment}
\begin{proposition} Let $-\Delta v + v = \mu \in H^{-1}(\mathbb{T}^2)$ with $v \equiv $ constant on supp $\mu$. Then $\mu \equiv $ constant on $\mathbb{T}^2$. \end{lem}
\begin{proof}
 Assume the constant is 1 without loss of generality. Then the above is the solution to the strictly convex variational problem
\begin{equation}
 \min_{\mu \in H^{-1}(\mathbb{T}^2)} \{ \iint G(x-y) d\mu(x) d\mu(y) : \int d\mu=1\}.
\end{equation}
By strict convexity there is only one critical point, which therefore must be $\mu \equiv $ constant on $\mathbb{T}^2$. 
\end{proof}
\end{comment}

We have the following interpretation of the divergence free condition when $\omega$ is a finite linear combination of
Dirac masses.

\begin{proposition}\label{omega2}  Let $-\Delta v + \kappa^2 v = \omega = \sum_{i=1}^d b_i \delta_{a_i}$ and assume that $T_{\omega}$ is divergence free in finite parts. Then setting $v(x) = \Phi(|x-a_i|) + H_i(x)$ it holds that
\[ \nabla H_i(a_i) = 0.\]
\end{proposition}

Before we continue with the proof of Proposition \ref{omega2}, we need the following Proposition which follows almost immediately from Proposition \ref{divfreenice}. The
proof is simple and contained in \cite{SS3}.

\begin{proposition}\label{divfree}
 If $X$ is divergence free in finite parts and is continuous in a neighborhood $U$ of the boundary of a smooth, compact set $K$ in $\Omega$, then
\[ \int_{\partial K} X \cdot \nu_K dS = 0.\]
\end{proposition}

\noindent \emph{Proof of Proposition \ref{omega2}:}
 We present the proof for $n=3$; the general case is similar. Assume that $\omega$ is a single Dirac mass at the origin with mass $4\pi$ without loss of generality. Then in spherical
coordinates we have
\begin{align}
 \nu = \frac{\partial}{\partial r} \;\;\; \tau = \frac{1}{r}\frac{\partial}{\partial \theta}\;\;\; \eta = \frac{1}{r \sin \theta}\frac{\partial}{\partial \varphi}.
\end{align}
We compute $T_{\omega} \cdot \nu$ in the $(\nu,\tau,\eta)$ basis to find
\begin{align}
 T_{\omega} \cdot \nu = \frac{1}{2}(( \partial_{\tau} v)^2 + (\partial_{\eta} v)^2 - (\partial_{\nu} v)^2 + v^2) \nu - (\partial_{\nu}v \partial_{\tau}v) \tau - (\partial_{\nu}v \partial_{\eta}v) \eta
\end{align}
Then we write $h = \Phi +H$ where $\Phi$ is the positive solution to $-\Delta \Phi = 4\pi \delta(x)$ and $H$ is smooth in a neighborhood of $0$. Then
we have $\partial_{\nu} \Phi = -\frac{1}{r^2} + o_{r}(1)$ as $r \to 0$ and $\partial_{\tau} \Phi = \partial_{\eta} \Phi = 0$. Thus as $r \to 0$ we have
\begin{align}
 T_{\omega} \cdot \nu = \frac{1}{2}\( - \frac{1}{r^4} + 2\frac{\partial_{\nu} H}{r^2}\)\nu + \(\frac{\partial_{\tau} H}{r^2}\)\tau + \(\frac{\partial_{\eta} H}{r^2}\)\eta.
\end{align}
Now using the fact that the integral of $\vec I(r)$ of $T_{\omega} \cdot \nu$ over $\partial B(0,r)$ is zero by Proposition \ref{divfree}, we have as $r \to 0$ that
\[ 0 = \nabla H (0) \cdot \vec I(r) = 4\pi |\nabla H(0)|^2 + o_r(1).\]
This implies $\nabla H(0)=0$. 
$\Box$

\section{Proof of Theorem \ref{rounddroplets}}
We are now ready to prove Theorem \ref{rounddroplets}. The main idea of the proof is simple. We use Theorem \ref{regularity} to write down the Euler-Lagrange equation satisfied
on the reduced boundary of $\{u^{\eps}=+1\}$. To leading order, the potential $v_{\eps}$ is constant on the boundary of an isolated droplet $\Omega_{j_i}$ whose center of mass
converges to $a_i$ (up to a subsequence), due to the logarithmic scaling of $G$ on $\mathbb{T}^2$ (cf. equation \eqref{greenlog}). The control of the isoperimetric deficit \eqref{isodef} controls the size of the error in making this approximation, and allows us
to conclude the curvature is asymptotically constant on the reduced boundary of droplets converging to $a_i$. 

\noindent \emph{Proof of Theorem \ref{rounddroplets}}:
By assumption, we have
\begin{equation}
 u^{\eps}(x) = -1 + 2\sum_{J=1}^{N(\eps)} \chi_{\Omega_j},
\end{equation}
where $N(\eps) = O(1)$ as $\eps \to 0$. We then apply Theorem \ref{regularity} to \eqref{E2a} to conclude that
\begin{align}\label{Eleqn2}
 \frac{\eps}{\sum_j |\Omega_j|} H_{\eps} - \frac{2\delta(\eps)}{\kappa^2\sum_j |\Omega_j|} + v_{\eps} = 0 \textrm{ on } \partial^* \{u^{\eps}=+1\},
\end{align}
holds for all $\eps > 0$. Since $N(\eps)=O(1)$ as $\eps \to 0$ and $P(\Omega_i) \to 0$ for each $i$, we conclude from compactness of $\mathbb{T}^2$ that the center of mass of each $\Omega_i$ converges
up to a subsequence to some $a_i$.  Let $J_i$ be the set of indices so that the center of mass of $\Omega_j$ converges to $a_i$.
\noindent We now expand the potential near $a_i$, first recalling that
\begin{align}
 v_{\eps}(x) = \int_{\mathbb{T}^2} G(x-y) \frac{ \sum_j \chi_{\Omega_j} (y)}{\sum_j |\Omega_j|} dy.
\end{align}
 Then, using \eqref{greenlog}, we have for $\eps$ sufficiently small, in a neighborhood
of $a_i$
\begin{align}
 v_{\eps}(x) = \int_{\mathbb{T}^2} -\frac{1}{2\pi} \log|x-y| \frac{ \sum_{j \in J_i} \chi_{\Omega_j} (y)}{\sum_j |\Omega_j|} dy + S_i^{\eps}(x),
\end{align}
where $S_i^{\eps}$ is uniformly bounded in $\eps$ in a neighborhood of $a_i$. Then letting $\bar x_j = P(\Omega_j)^{-1} x$, $\bar y_j = P(\Omega_j)^{-1}y$, $v_{\eps}$ in these variables becomes
\begin{align}
  v_{\eps}(x) &= \int_{\mathbb{T}^2} -\frac{1}{2\pi} \log|x-y| \frac{ \sum_{j \in J_i} \chi_{\Omega_j} (y)}{\sum_j |\Omega_j|} dy + S_i^{\epsilon}(x)\\
&= \frac{\sum_{j \in J_i} -\frac{1}{2\pi} \log P(\Omega_j) | \Omega_j|}{\sum_j |\Omega_j|} + \frac{\sum_j P(\Omega_j)^2 }{\sum_j |\Omega_j|} \int_{\bar \Omega_j} -\frac{1}{2\pi} \log|\bar x_j - \bar y_j|d\bar y_j + S_i^{\eps}(x).\label{A8} 
\end{align}
We then use the inequality
\begin{equation}\label{diamcontrol}
 \textrm{essdiam} (\bar \Omega_j) \leq \frac{1}{2} P(\bar \Omega_j)=\frac{1}{2},
\end{equation}
which follows (for instance) from \cite[Theorem 7 and Lemma 4] {ambrosio2} noting that in view of \cite[Proposition 6(ii)]{ambrosio2} it suffices
to consider only simple sets \cite[Definition 3]{ambrosio2}. Thus we have from \eqref{diamcontrol} and the defintion of $\bar x_j$, $\bar y_j$
\begin{equation}\label{A7}
 \left|\int_{\bar \Omega_j} -\frac{1}{2\pi} \log|\bar x_j - \bar y_j|d\bar y_j \right| \leq \frac{1}{2\pi} \int_{B(0,2)} |\log |x|| dx \leq C.
\end{equation}
 where $C>0$ is independent of $\eps$. Inserting \eqref{A7} into \eqref{A8} and using the bound on the isoperimetric defecit \eqref{isodef} we have, using the fact that $P(\{u^{\eps}=+1\}) \to 0$ as $\eps \to 0$ (cf. equation \eqref{apriori1}), that
for any $k \in J_i$ and $x_i^{\eps} \in \bigcup_{j \in J_i} \partial^* \Omega_{j_i}$
\[\frac{v_{\eps}(x_i^{\eps})}{\sum_{j}\log P(\Omega_j) } = -\frac{1}{2\pi} \frac{\sum_{j \in J_i}\log P(\Omega_j) | \Omega_j|}{\sum_{j}\log  P(\Omega_j)  \sum_j |\Omega_j|}+ o_{\eps}(1).\]

\noindent Rewriting the Euler-Lagrange equation \eqref{Eleqn2} we have
\begin{align}
 \left \|\frac{-\eps H_{\eps}}{\sum_{j}\log P(\Omega_j)  \sum_j |\Omega_j|} +c_i^{\eps}\right\|_{L^{\infty}(\bigcup_{j \in J_i} \partial^* \Omega_j)} \to 0 \textrm{ as } \eps \to 0,\end{align}
where \begin{align}
c_i^{\eps} = -\frac{1}{2\pi} \frac{\sum_{j \in J_i}\log P(\Omega_j) | \Omega_j|}{\sum_{j}\log  P(\Omega_j) \sum_j |\Omega_j|} + \frac{2}{\kappa^2}\frac{\delta(\eps)}{\sum_{j}\log  P(\Omega_j) \sum_j |\Omega_j|}.
\end{align}
Now choose a subsequence $\eps_k$ so that the $\liminf$ in the definition of $\bar \delta$ (cf. \eqref{bardelta}) is achieved as $\eps_k \to 0$. It is clear that the first term in the definition of $c_{i}^{\eps}$ is bounded uniformly and positive as $\eps \to 0$, and therefore converges subsequentially to some $c_i^0 \geq 0$.  Therefore we have (possibly taking a further subsequence) that
\[ c_i^{\eps_k} \to c_i^0 - \frac{2}{\kappa^2} \bar \delta,\]
as $\eps_k \to 0$. Choosing $\bar \delta_{cr} = \frac{ \kappa^2 c_i^0}{2}$, we obtain the result.
 %Now choose $k = \textrm{argmax}_{j \in J_i} P(\Omega_j)$,
$\Box$

\section{The diffuse interface energy}\label{sec5}
In this section we study
\begin{align}
  \label{EE2}
 \mathcal{E}[u] = \int_\Omega \( \frac{\eps^2}{2} |\nabla u|^2 +
  V(u)\) dx + \frac{1}{2} \int_\Omega \int_\Omega (u(x)- \bar u)
  G_0(x, y) (u(y)- \bar u) \, dx \, dy.
\end{align}
We make the particular choice of $V(u) = \frac{1}{4} (1-u^2)^2$, but our results will hold, with minor adjustments to the proofs,
under general assumptions on $V$. Recalling the discussion in Section \ref{diffusesection} we know that any stationary point $u^{\eps}$ of \eqref{EE2} in the sense of Definition \ref{defcp} is a critical point in the sense of Definition \ref{diffusestationary} (\textbf{??}), which is easily seen to be a solution to
\begin{equation}\label{ELdiffuse} -\frac{\eps^2}{\delta(\eps)} \Delta u^{\eps} - \frac{1}{\delta(\eps)} u^{\eps}(1-(u^{\eps})^2) + v_{\eps} = \lambda_{\eps},\end{equation}
where $\lambda_{\eps}$ is the Lagrange multiplier arising from the volume constraint and
\[ v_{\eps}(x) = \int_{\mathbb{T}^n} G(x-y) \frac{1+u^{\eps}(y)}{\delta(\eps)} dy.\]
\noindent We recall our main assumption that
\begin{equation}\label{apriori20} \limsup_{\eps \to 0} \mathcal{H}^{n-1}\(  \{u^{\eps} \geq -1 + \delta(\eps)^{1+\alpha}\}\) = 0 \textrm{ for some } \alpha > 0.\end{equation}
Our methods will be very
similar to those of the sharp interface energy \eqref{E2}, and follow closely the methods of \cite{SS5,SS3} for Ginzburg-Landau. In particular, we first show that \eqref{ELdiffuse}
is equivalent to a certain 2-tensor $S_{\eps}=\{S_{ij}\}$ having zero divergence (cf. Proposition \ref{divfreeS} below). We then use \eqref{apriori20} to cover the set where $u^{\eps}$ is close to $+1$ by balls whose boundaries have very small $\mathcal{H}^{n-1}$ measure (cf. Proposition \ref{sumballs}). Finally we show that
away from the set where $u^{\eps}$ is close to $+1$, $S_{\eps}$ is close in $L^1$ to the tensor $T_{\eps}=\{T_{ij}\}$ defined by
\begin{equation}\label{Tdefdiffuse}
 T_{ij} = -\partial_i v_{\eps} \partial v_{\eps} + \frac{1}{2} \delta_{ij} |\nabla v_{\eps}|^2.
\end{equation}
 We begin by observing that if $u^{\eps}$ solves \eqref{ELdiffuse} then it holds by direct computation that
\begin{align}
 \div S^{\eps}  &= 0,
\end{align}
where
\begin{multline}\label{Sdiffusedef}
S_{ij}^{\eps} = \partial_i v^{\eps} \partial_j v^{\eps} - \frac{\eps^2}{\delta(\eps)^2}\partial_i u^{\eps} \partial_j u^{\eps} + \frac{\delta_{ij}}{2} \( \frac{\eps^2}{\delta(\eps)^2} |\nabla u^{\eps}|^2  + \frac{1}{4 \delta(\eps)^2}(1-|u^{\eps}|^2)^2 - |\nabla v^{\eps}|^2 - \frac{(u^{\eps}+1)}{\delta(\eps)} \lambda_{\eps} \) \\ + \delta_{ij}\frac{v^{\eps}(u^{\eps} +1)}{\delta(\eps)}.
\end{multline}
This is summarized in the following proposition. 
\begin{proposition}\label{divfreeS}
 Let $u^{\eps}$ be a solution to \eqref{ELdiffuse}. Then 
\[ \div S^{\eps} = 0 \textrm{ in } \mathcal{D}'(\mathbb{T}^n),\]
where $S_{\eps}$ is given by \eqref{Sdiffusedef}.
\end{proposition}
\begin{proof}
 A direct computation using the fact that $u^{\eps} \in C^{2}(K)$ for any $K \subset \subset \mathbb{T}^n$ yields
\[ \div S^{\eps} = \nabla u^{\eps} \( -\frac{\eps^2}{\delta(\eps)} \Delta u^{\eps}  + \frac{1}{\delta(\eps)} u^{\eps} (1-(u^{\eps})^2) + v_{\eps}-\lambda_{\eps} \)=0.\]
\end{proof}

\begin{proposition}
 Let $\{u_{\eps}\}_{\eps}$ be a sequence of solutions to \eqref{ELdiffuse} satisfying \eqref{apriori20} and 
\begin{equation}
\limsup_{\eps \to 0} |\lambda_{\eps}| < +\infty.\end{equation} 
  For any $\eps >0$ define the 2-tensors $S_{\eps}$ as above and $T_{\eps}$ by
\begin{align}
 T_{ij}^{\eps} = \partial_i v^{\eps} \partial_j v^{\eps} - \frac{\delta_{ij}}{2}|\nabla v^{\eps}|^2.
\end{align}
Then $T_{\eps} - S_{\eps}$ tends to $0$ in $L_{\delta}^1(\mathbb{T}^n)$. 
\end{proposition}

\begin{proof}
We once again argue as in \cite{SS3} for Ginzburg-Landau. From \eqref{apriori20} and Proposition \ref{sumballs}, the set of $x$ in $\mathbb{T}^n$ such that $u(x) \geq -1+\delta(\eps)^{1+\alpha}$ can be covered by a collection
of balls $B_1, \cdots, B_k$ such that
\[ \sum_{i=1}^k r(B_i)^{n-1} \leq C\mathcal{H}^{n-1}(\{u^{\eps} \geq -1 + \delta(\eps)^{1+\alpha}\}).\]
 We denote $Z_{\eps}$ as the union of these balls and and observe that
\[ \lim_{\eps \to 0} \textrm{Cap}_1(Z_{\eps}) = 0.\]
This follows from the fact that the 1-capacity of a ball $B(x,r)$ is $\alpha_{n-1}r^{n-1}$ and the capacity is subadditive so $\textrm{Cap}_1(Z_{\eps}) \leq C\mathcal{H}^{n-1}(\{u^{\eps} \geq -1 + \delta(\eps)^{1+\alpha}\})$, which tends to zero by assumption. The difference between $S^{\eps}$ and $T^{\eps}$ is
\begin{multline} S^{\eps} - T^{\eps}  =  - \frac{\eps^2}{\delta(\eps)^2}\partial_i u^{\eps} \partial_j u^{\eps} + \frac{\delta_{ij}}{2} \( \frac{\eps^2}{\delta(\eps)^2} |\nabla u^{\eps}|^2  + \frac{1}{4 \delta(\eps)^2}(1-|u^{\eps}|^2)^2  - \frac{1}{\delta(\eps)}(u^{\eps}+1) \lambda_{\eps} \) \\ + \delta_{ij}\frac{v_{\eps}(u^{\eps}  +1)}{\delta(\eps)}.
\end{multline}
Thus it is easily seen that
\begin{equation}\label{STDiff}
 |S^{\eps} - T^{\eps}| \leq C\( \frac{\eps^2}{\delta(\eps)^2} |\nabla u^{\eps}|^2  + \frac{1}{2 \delta(\eps)^2}(1-|u^{\eps}|^2)^2  + \frac{2}{\delta(\eps)} (|v_{\eps}|+|\lambda_{\eps}|)|u^{\eps} +1|\).
\end{equation}
Now define the function $\chi:[0,1] \to [0,1]$ to be the affine interpolation between the values $\chi(-1)=1$, $\chi(-1+\delta(\eps)^{1+\alpha})=1$ and $\chi(-1/2)=1/2$ and $\chi(0)=0$ and $\chi(1)=-1$. Multiply \eqref{ELdiffuse} by $\chi(u^{\eps}) + u^{\eps}$ and integrating by parts we have
\begin{equation}\label{R1}
\frac{1}{\delta(\eps)^2} \int_{\mathbb{T}^n} \eps^2 |\nabla u^{\eps}|^2 (\chi'(u^{\eps}) + 1) - u^{\eps}(\chi(u^{\eps}) + u^{\eps}) (1-(u^{\eps})^2) = -\frac{1}{\delta(\eps)}\int_{\mathbb{T}^n} (\chi(u^{\eps})+u^{\eps}) (v^{\eps} - \lambda_{\eps}) dx.
\end{equation}

The set $\{\chi(u^{\eps})=1\}$ contains the set $u^{\eps}(x) \leq -1 + \delta(\eps)^{1+\alpha}$ and therefore $Z_{\eps}^c$. When $|u^{\eps}| \geq 1/2$, which is true on $\{\chi(u^{\eps})=1\}$,
the left side of \eqref{R1} can be bounded from below by \[\frac{1}{2\delta(\eps)^2} \int_{Z_{\eps}^c} \eps^2 |\nabla u^{\eps}|^2 + \frac{1}{4}(1-(u^{\eps})^2)^2.\]
%where $\mathcal{E}_{loc}^{\eps}(u^{\eps})$ denotes the local part of $\mathcal{E}^{\eps}%(u^{\eps})$. 
Indeed on $\{\chi(u^{\eps})=1\}$ we have\[ |u^{\eps}|(1-|u^{\eps}|) \geq (1-|u^{\eps}|^2),\]
 since $|u^{\eps}| \geq (1+|u^{\eps}|)$ when $u^{\eps} \in (-1,-1/2)$.
%The last term in the above expression is $o_{\eps}(1)$ as $\eps \to 0$ by our assumption %that $\delta(\eps)^{-2}\mathcal{E}_{loc}^{\eps}(u^{\eps})$ is bounded independently of %$\eps$.
Since $(u^{\eps}+1)/\delta(\eps)$ is bounded in $(C^0(\mathbb{T}^n))^*$ and therefore in $W^{-1,p}$ for $p \in (1,n/(n-1))$ by standard embeddings, we conclude that $v_{\eps}$ is bounded uniformly in $L^1$. Then by the definition of $\chi$ and the fact that $u^{\eps} \in (-1,-1+\delta(\eps)^{1+\alpha})$ where $\chi(u^{\eps}) = 1$ we have
\begin{align}
\frac{1}{\delta(\eps)}\left|\int_{\mathbb{T}^n} (\chi(u^{\eps})+u^{\eps})v_{\eps}\right| =\frac{1}{\delta(\eps)}\left|\int_{Z_{\eps}^c} (1+u^{\eps})v_{\eps}\right| \leq C\delta(\eps)^{\alpha} \|v_{\eps}\|_{L^1}.
\end{align}

\begin{comment}
\begin{align}
 \frac{1}{\delta(\eps)} \int_{Z_{\eps}^c} v_{\eps}(u^{\eps} - \bar u^{\eps}) =  \int_{Z_{\eps}^c} v_{\eps} + \int_{Z_{\eps}^c} (u^{\eps}+1) v_{\eps}.
\end{align}
Then on $\mathbb{T}^n$ we have $0= \int_{\mathbb{T}^n} v_{\eps} = \int_{Z_{\eps}} v_{\eps} + \int_{Z_{\eps}^c} v_{\eps}$.
\end{comment}
\noindent Combining the above we conclude
\begin{equation}\label{STDiff1}
 \frac{1}{\delta(\eps)^2} \int_{Z_{\eps}^c} \eps^2 |\nabla u^{\eps}|^2 + (1-(u^{\eps})^2)^2 = o_{\eps}(1) \textrm{ as } \eps \to 0.
\end{equation}
Focusing on the remaining terms in \eqref{STDiff}, it remains to show that
\begin{equation}\frac{1}{\delta(\eps)}\int_{Z_{\eps}^c}|1+u^{\eps}|( |v_{\eps}|+|\lambda_{\eps}|) = o_{\eps}(1) \textrm{ as } \eps \to 0.\end{equation}
 This however follows from the definition of $Z_{\eps}^c$:
\begin{equation}\label{R4}  \frac{1}{\delta(\eps)}\int_{Z_{\eps}^c} |u^{\eps}+1| (|v_{\eps}+|\lambda_{\eps}|) \leq C(\|v_{\eps}\|_{L^1}+|\lambda_{\eps}|) \delta(\eps)^{\alpha} \leq C\delta(\eps)^{\alpha},\end{equation}
 where we've used the fact that $\limsup_{\eps} |\lambda_{\eps}| < +\infty$. 
Finally combining \eqref{R4} and \eqref{STDiff1} and using \eqref{STDiff} we conclude that 
\[\int_{\mathbb{T}^n \backslash Z_{\eps}} |T_{\eps} - S_{\eps}| \to 0 \textrm{ as } \eps \to 0,\]
the desired result.\end{proof}

\noindent We now complete the proof of Theorem \ref{main3}.
\begin{proof}
Choose a decreasing subsequence $\{\eps_k\}$ tending to zero such that $\sum_k \textrm{Cap}_1(Z_{\eps_k}) < +\infty$ and let
\[ E_{\delta} = \bigcup_{k > \frac{1}{\delta}} Z_{\eps_k}.\]

\noindent% This is indeed possible from \eqref{apriori1} and the fact that $\mathcal{H}^{n-1}(\{u^{\eps} \geq -1 + \eps^{\alpha}\}) \to 0$. 
Since $\omega_{\eps} := (1+u^{\eps})/\delta(\eps)$ is a family of probability measures on $\mathbb{T}^n$, we have $\omega_{\eps} \to \omega$ weakly in $(C^0(\mathbb{T}^n))^*$ up to a subsequence, and thus $\omega_{\eps} \to \omega$ strongly
in $W^{-1,p}$ for $p \in (1,n/(n-1))$ via the compact embedding $(C^0(\mathbb{T}^n))^* \subset \subset W^{-1,p}$ which follows from the
compact embedding $W^{1,q}(\mathbb{T}^n) \subset \subset C^0(\mathbb{T}^n)$ for $q > n/(n-1)$.
Now define 
\begin{equation}\label{deltadef} F_{\delta} := E_{\delta} \cup \tilde E_{\delta},\end{equation}
where $\tilde E_{\delta}$ are the sets given by Proposition \ref{msrconv} with $\kappa=0$. Then by subadditivity of capacity \cite{Gariepy} we have 
\[\lim_{\delta \to 0} \textrm{ Cap}_1(F_{\delta}) = 0.\]

 From Proposition \ref{msrconv} we therefore conclude
\[ \nabla v_{\eps} \to \nabla v \textrm{ in } L_{\delta}^2(\mathbb{T}^n).\]
 Thus, combining the above, we have
\[ S_{\eps} - T_{\omega} \textrm{ converges to } 0 \textrm{ in } L_{\delta}^1(\mathbb{T}^n),\]
where the sets $F_{\delta}$ in Definition \ref{divfreefin} are given by \eqref{deltadef}. Thus $T_{\omega}$ is divergence free in finite part from
Proposition \ref{divfree}. 
\end{proof}

\noindent \textbf{Acknowledgments} The research of the author was
partially supported by NSF research grant DMS-0807347 and by the Herchel Smith fellowship
at the University of Cambridge. The author would like to thank his
advisor Sylvia Serfaty for suggesting the problem and offering helpful suggestions and comments throughout. The author
would also like to thank and Alexander Volkmann, Theodora Bourni and Robert Haslhofer
for helpful discussions throughout the course of this work.


\begin{thebibliography}{10}

%\bibitem{acerbi} E.\ Acerbi, N.\ Fusco, M. \ Morini. \newblock
%  Minimality via second variation for a nonlocal isoperimetric
%  problem, Preprint: http://cvgmt.sns.it/paper/540/.

\bibitem{ACO} G.\ Alberti, R.\ Choksi and F.\ Otto.  \newblock Uniform
  Energy Distribution for an Isoperimetric Problem With Long-range
  Interactions.  \newblock {\em Journal Amer. Math. Soc.}, 2:569-605,
  2010.


%\bibitem{ACP} R. Alicandro, M. Cicalese, M. Ponsiglione. Variational
%  equivalence between Ginzburg-Landau, XY spin systems and screw
%  dislocations energies. To appear in {\em Indiana Univ. Math. J.}

\bibitem{Allard} W.\ Allard.  \newblock A regularity theorem for the first variation of the area integrand.  \newblock {\em Bull. Amer. Math. Soc.}, 77:772-776,
  1971.


\bibitem{ambrosio2} L.\ Ambrosio, V.\ Caselles, S.\ Masnou and
  J.\ Morel.  \newblock Connected components of sets of finite
  perimeter with applications to image processing.  \newblock {\em
    J. Eur. Math. Soc.}, 3:39-92, 2001.
\bibitem{ambrosio3} L.\ Ambrosio and E.\ Paolini. Partial regularity for quasi minimizers of perimeter. \newblock {\em Ric. Mat.}, 48:167-186,1998.

\bibitem{Aydi} H. Aydi.  \newblock Lines of vortices for solutions of the Ginzburg-Landau equations.\newblock {\em J. Math. Pures. Appl.}, 89:49-69,
  2008.
  
\bibitem{BBH}F. \ Bethuel, H.\ Br\'{e}zis and F.\ H\'{e}lein. {\em
    Ginzburg-Landau Vortices}. Birkhauser Progress in Non. Partial
  Diff. Eqns and Their Appns. \textbf{70}, (1994)

\bibitem{Majda} A.\ Bertozzi and A.\ Majda. {\em Vorticity and Incompressible Flow}. Cambridge Texts in Applied Mathematics. Cambridge University Press, (2002).

%\bibitem{bonnesen24} 
%T.~Bonnesen.
%\newblock \"{U}ber das isoperimetrische {D}efizit ebener {F}iguren.
%\newblock {\em Math. Ann.}, 91:252--268, 1924.

\bibitem{Braides} A.\ Braides. {\em Gamma-Convergence for Beginners}.
  Oxford Lecture Series in Math., (2002).

%\bibitem{braides08}
%A.~Braides, L.~Truskinovsky.
%\newblock Asymptotic expansions by {$\Gamma$}-convergence.
%\newblock {\em Continuum Mech. Thermodyn.}, 20:21--62, 2008.

%\bibitem{brezis79} H.~Br{\'e}zis, F.~Browder.  \newblock A property
%  of {S}obolev spaces.  \newblock {\em Comm. Partial Differential
%    Equations}, 4:1077--1083, 1979.
\bibitem{Chemin} J.\ Chemin. {\em Fluid Parfaits Incompressibles}. Soci\'{e}t\'{e} Math\'{e}matique de France. Institut Henri Poincar\'{e}, (1994).


%\bibitem{CO}X.\ Chen, Y.\ Oshita. An Application of the Modular Function
%  in Nonlocal Variational Problems, {\it Arch. Rational Mech. Anal.}
%  186:109--132, 2007.


\bibitem{choksi10} R.\ Choksi and M.\ Peletier.  \newblock Small
  volume fraction limit of the diblock copolymer problem: {I.  Sharp}
  interface functional.  \newblock {\em SIAM J. Math. Anal.},
  42:1334--1370, 2010.

\bibitem{choksi11} R.\ Choksi and M.\ Peletier.  \newblock Small
  volume fraction limit of the diblock copolymer problem: {II.
    Diffuse} interface functional.  \newblock {\em SIAM
    J. Math. Anal.}, 43:739--763, 2011.


%\bibitem{Choksi2} R. \ Choksi, M.\ Peletier, J. F. Williams. On the
%  Phase Diagram for Microphase Separation of Diblock Copolymers: An
%  Approach via a Nonlocal Cahn--Hilliard Functional.  \newblock{ \em
%    SIAM Journal of Applied Mathematics,} 69:1712--1738, 2009.
\bibitem{choksi12}
R.\ Choksi and P.\ Sternberg.
\newblock On the first and second variations of a non-local isoperimetric problem.
\newblock {\em J. Reigne angew. Math.}, 611:75--108, 2007.
\bibitem{Cicalese} M. \ Cicalese and E.\ Spadaro. \newblock Droplet Minimizers of
  an Isoperimetric Problem with long-range interactions. Preprint:
  http://arxiv.org/abs/1110.0031

\bibitem{delort} J.\ Delort. \newblock Existence de nappes de tourbillon en dimension deux.
\newblock{\em J. Amer. Math. Soc.}, 4:553-586, 1991.

\bibitem{pernamajda} R.\ DiPerna and A.\ Majda. \newblock Concentrations in regularizations for 2-D incompressible flow.
\newblock{\em Comm. on Pure and Appl. Math}, 40:301-345, 1987.

\bibitem{pernamajda2} R.\ DiPerna and A.\ Majda. \newblock Reduced Hausdorff dimension and concentration-cancellation for two-dimensional incompressible flow,
\newblock{\em J. Amer. Math. Soc.}, 1:59--95, 1988.


\bibitem{Gariepy} C.\ Evans and R.\ Gariepy. {\em Measure Theory and Fine Properties of Functions}. Studies in Advanced Mathematics. CRC Press, Boca Raton, FL, (1992).

\bibitem{degennes79} P.\ {de Gennes}. \newblock Effect of
  cross-links on a mixture of polymers. \newblock {\em J. de Physique
    -- Lett.}, 40:69--72, 1979.

%\bibitem{fusco08} N.~Fusco, F.~Maggi, A.~Pratelli.  \newblock The
%  sharp quantitative isoperimetric inequality.  \newblock {\em Ann. of
%    Math.}, 168:941--980, 2008.
\bibitem{giaquinta} M.\ Giaquinta. {\em Introduction to Regularity Theory for Nonlinear Elliptic Systems}. Birkh\"{a}user Lectures in Mathematics ETH Z\"{u}rich, (1994).
  
\bibitem{trudinger}
D.\ Gilbarg and N.\ Trudinger.\newblock {\em Elliptic Partial Differential Equations of Second Order}.
\newblock Springer-Verlag, Berlin, (1983).

\bibitem{glotzer95} S.~Glotzer, E.~A. Di~Marzio, M.~Muthukumar.
  \newblock Reaction-controlled morphology of phase-separating
  mixtures.  \newblock {\em Phys. Rev. Lett.}, 74:2034--2037, 1995.

\bibitem{giusti}
E.\ Giusti.\newblock{\em Minimal Surfaces and Functions of Bounded Variation}.\newblock Birkh\"{a}user Boston, (1984).


\bibitem{gms10b} D.\ Goldman, C.\ Muratov and S.\ Serfaty. \newblock
  The {$\Gamma$}-limit of the two-dimensional {Ohta-Kawasaki} energy.
  {I. Droplet} density. \newblock{\em Arch. Ration. Mech. Anal.}, 210:581--613, 2013.
\bibitem{gms11b} D.\ Goldman, C.\ Muratov and S.~Serfaty. \newblock
  The {$\Gamma$}-limit of the two-dimensional {Ohta-Kawasaki} energy.
  {II. Droplet arrangement at the sharp interface level via the renormalized energy } density. \newblock{\em Arch. Ration. Mech. Anal.}, 212:445--501, 2014.

\bibitem{GV} D.\ Goldman and A.\ Volkmann. \newblock
On the regularity of stationary points of a non-local isoperimetric problem. \newblock (Preprint available at http://arxiv.org/abs/1405.4550).
%\bibitem{jerrardball} R.\ Jerrard. Lower bounds for generalized
%  Ginzburg-Landau functionals \newblock{\em SIAM J. Math. Anal},
%  30:721--746, 1999.

%\bibitem{lieb-loss} E.~H. Lieb, M.~Loss.  \newblock {\em Analysis}.
%  \newblock Amer. Math. Soc., (2001).
\bibitem{NamLe}
N.\ Le. \newblock Regularity and nonexistence results for some free-interface problems related to Ginzburg-Landau vortices. \newblock {\em Interfaces Free Bound.}, 11:139-152, 2009.

\bibitem{lundqvist} S.\ Lundqvist,  N.\ March and editors. \newblock
  {\em Theory of inhomogeneous electron gas}.  \newblock Plenum Press,
  New York, (1983).


\bibitem{maggi} F.\ Maggi. \newblock
  {\em Sets of Finite Perimeter and Geometric Variational Problems: An Introduction to Geometric Measure Theory}.  \newblock Cambridge University Press,
  New York, (2012).

%\bibitem{mu} S.~M\"uller.  \newblock Singular perturbations as a
%  selection criterion for periodic minimizing sequences.  \newblock
%  {\em Calc. Var. Part. Dif.}, 1:169--204, 1993.

\bibitem{massari}
U.\ Massari
\newblock Esistenza e regolorit\'{a} delle ipersurfice di curvutura media assegnata in $\mathbb{R}^n$. \newblock {\em Arch. Rat. Mech. Anal.}, 55:357-382, 1974

\bibitem{m:phd} C.\ Muratov. \newblock {\em Theory of domain
    patterns in systems with long-range interactions of Coulombic
    type}. \newblock Ph. D. Thesis, Boston University, (1998).

\bibitem{m:pre02} C.\ Muratov. \newblock Theory of domain patterns
  in systems with long-range interactions of {Coulomb} type.
  \newblock {\em Phys. Rev. E}, 66:1--25, 2002.

\bibitem{m:cmp10} C.\ Muratov. \newblock Droplet phases in
  non-local {Ginzburg-Landau} models with {Coulomb} repulsion in two
  dimensions.  \newblock {\em Comm. Math. Phys.}, 299:45--87, 2010.

\bibitem{nyrkova94} I.\ Nyrkova, A.\ Khokhlov and M.\ Doi. \newblock Microdomain structures in polyelectrolyte systems:
  calculation of the phase diagrams by direct minimization of the free
  energy.  \newblock {\em Macromolecules}, 27:4220--4230, 1994.

\bibitem{ohta86} T.\ Ohta and K.\ Kawasaki. \newblock Equilibrium
  morphologies of block copolymer melts.  \newblock {\em
    Macromolecules}, 19:2621--2632, 1986.

%\bibitem{ortner12} C.~Ortner, E.~S\"uli.  \newblock A note on
%  linear elliptic systems on {$\mathbb R^d$}.  \newblock
%  arXiv:1202.3970v3, 2012.

%\bibitem{osserman79}

%R.~Osserman.
%\newblock Bonnesen-style isoperimetric inequalities.
%\newblock {\em Amer. Math. Monthly}, 86:1--29, 1979.


%\bibitem{ren00trusk} X.~Ren and L.~Truskinovsky.  \newblock Finite
%  scale microstructures in nonlocal elasticity.  \newblock {\em
%    J. Elasticity}, 59:319--355, 2000.

\bibitem{CPoints} M.\ R\"{o}ger and Y.\ Tonegawa. \newblock Convergence of the phase-field approximations to the Gibbs-Thompson law. \newblock{\em Calc. Var. \& PDE},
  32:111--136, 2008.

%\bibitem{compagnon} E.\ Sandier, S.\ Serfaty. Improved Lower Bounds
%  for Ginzburg-Landau Energies via Mass Displacement. \newblock{\em Analysis \& PDE},
%  4-5:757--795, 2011.

%\bibitem{sandierball} E.\ Sandier. Lower bounds for the energy of unit
%  vector fields and applications, \newblock {\em J. Funct. Anal.,}
%  152:379--403, 1998

\bibitem{SS3} E. \ Sandier and S.\ Serfaty. {\em Vortices in the Magnetic
    Ginzburg-Landau Model}. Birkh\"{a}user Progress in Non. Partial
  Diff. Eqns and Their Appns. \textbf{70}, (2007)

%\bibitem{SS2} E.\ Sandier, S.\ Serfaty. From Ginzburg Landau to Vortex
%  Lattice Problems, {\em Comm. Math. Phys.}, 313:635-743, 2012.

%\bibitem{SS1} E.\ Sandier, S.\ Serfaty. A rigorous derivation of a
%  free boundary problem arising in superconductivity, \newblock{\em
%    Ann. Sci. de l'ENS,} 33:561--592, 2000.

%\bibitem{SS4} E. ~Sandier, S.~Serfaty. 2D Coulomb gases and the
%  Renormalized Energy. Preprint  {\tt arXiv:1201.3503}.

\bibitem{SS5} E.\ Sandier and S.\ Serfaty. \newblock Limiting Vorticities for the Ginzburg-Landau Equations. \newblock{\em
   Duke Math J,} 117:403--446, 2003.

\bibitem{sang} L.\ Serge. {\em Analysis II.} Addison-Wesley, (1969)

\bibitem{Simons} L.\ Simon. {\em Lectures on geometric measure theory}. Volume 3 in the Proceedings of the Centre for Mathematical and Analysis, (1984)

\bibitem{spadaro09} E.\ Spadaro.\newblock Uniform energy and density
  distribution: diblock copolymers' functional.  \newblock {\em
    Interfaces Free Bound.}, 11:447--474, 2009.

%\bibitem{sternberg11} P.~Sternberg, I.~Topaloglu.  \newblock A note
%  on the global minimizers of the nonlocal isoperimetric problem in
%  two dimensions.  \newblock {\em Interfaces Free Bound.},
%  13:155--169, 2010.

\bibitem{sternberg}
P.~Sternberg, I.~Topaloglu.
\newblock On the global minimizers to a non-local isoperimetric problem in two dimensions. 
\newblock {\em Interfaces Free Bound.}, 13:155--169, 2011.

\bibitem{stillinger83} F.\ Stillinger. \newblock Variational model
  for micelle structure.  \newblock {\em J. Chem. Phys.},
  78:4654--4661, 1983.
  
  \bibitem{tamanini} I.\ Tamanini. \newblock Boundaries of Caccioppoli sets with H\"{o}lder continuous normal vector.  \newblock {\em J. Reine Angew. Math.}, 334:27-39, 1982.
  
\bibitem{volkmann} A.\ Volkmann. {\em Regularity of isoperimetric
hypersurfaces with obstacles in
Riemannian manifolds}. Diploma Thesis, (2010). \url{http://www.aei.mpg.de/~volkmann/dipl.pdf}
%\bibitem{tinkham} M. Tinkham. {\em Introduction to
%    superconductivity. Second edition,}  McGraw-Hill, New York, (1996).

%\bibitem{yip06} N.~K. Yip.  \newblock Structure of stable solutions of
%  a one-dimensional variational problem.  \newblock {\em ESAIM Control
%    Optim. Calc. Var.}, 12:721--751, 2006.


\bibitem{zheng} Y.\ Zheng. \newblock Concentration-cancellation for the velocity fields in two dimensiona incompressible fluid flows.  \newblock {\em Comm. Math. Phys.}, 135:581--594, 1991.

\end{thebibliography}
\end{document}